\documentclass[12pt]{article}
\usepackage[english]{babel}
\usepackage{epsfig}
\usepackage{amsfonts}
\usepackage{amssymb}
\usepackage{amsmath}
\usepackage{amsthm}
\usepackage{latexsym}
\usepackage{graphicx}
\usepackage{bm}
\usepackage{tabularx}
\usepackage{booktabs} 
\usepackage{enumerate}
\usepackage{caption}
\usepackage{wrapfig}

\usepackage[titletoc]{appendix}

\usepackage[numbers]{natbib}

\usepackage{bbm} 
\usepackage{url}
\usepackage{subfig}
\usepackage{hyperref}
\usepackage{color}
\usepackage{float}

\catcode`\@=11

\newcommand{\longdot}[1]{\raisebox{-1.4pt}{$\stackrel{\bullet}{#1}$}}
\newcommand{\shortdot}[1]{\raisebox{-0.4pt}{$\stackrel{\bullet}{#1}$}}

\theoremstyle{plain}
\newtheorem{theorem}{Theorem}[section]

\newtheorem{proposition}[theorem]{Proposition}

\theoremstyle{definition}

\theoremstyle{remark}

\begin{document}
\title{Limiting the Oscillations in Queues with Delayed Information Through a Novel Type of Delay Announcement}
\author{ 
  Sophia Novitzky \\ Center for Applied Mathematics \\ Cornell University
\\ 657 Rhodes Hall, Ithaca, NY 14853 \\  sn574@cornell.edu  \\ 
 \and 
  Jamol Pender \\ School of Operations Research and Information Engineering \\ Cornell University
\\ 228 Rhodes Hall, Ithaca, NY 14853 \\  jjp274@cornell.edu  \\ 
 \and  
Richard H. Rand \\ Sibley School of Mechanical and Aerospace Engineering \\ Department of Mathematics \\ Cornell University
\\ 535 Malott Hall, Ithaca, NY 14853 \\  rand@math.cornell.edu  \\ 
 \and  
Elizabeth Wesson \\ Center for Applied Mathematics \\ Cornell University
\\Rhodes Hall 657, Ithaca, NY 14853 \\  enw27@cornell.edu  \\ 
 }

\maketitle
\begin{abstract}
Many service systems use technology to notify customers about their expected waiting times or queue lengths via delay announcements.  However, in many cases, either the information might be delayed or customers might require time to travel to the queue of their choice, thus causing a lag in information. In this paper, we construct a neutral delay differential equation (NDDE) model for the queue length process and explore the use of \emph{velocity} information in our delay announcement.  Our results illustrate that using velocity information can have either a beneficial or detrimental impact on the system. Thus, it is important to understand how much velocity information a manager should use. In some parameter settings, we show that velocity information can eliminate oscillations created by delays in information. We derive a fixed point equation for determining the optimal amount of velocity information that should be used and find closed form upper and lower bounds on its value. When the oscillations cannot be eliminated altogether, we identify the amount of velocity information that minimizes the amplitude of the oscillations. However, we also find that using too much velocity information can create oscillations in the queue lengths that would otherwise be stable.\\



\end{abstract}

\section{Introduction}
Many corporations and services eagerly adopt new technologies that allow the service managers to interact with their customers. One highly important aspect of the communication is the delay announcement, which informs the customers of their estimated waiting time or queue length. Delay announcements are used in variety of service systems. For example, some hospital emergency rooms display their expected waiting times online. Telephone call centers warn the customers that are placed on hold about extensive waiting times. Amusement parks update the waiting times for different rides, and some transportation networks warn about heavy traffic or delays via road signs. 

The popularity of delay announcements among service managers extends beyond customer satisfaction. Informing customers about their waiting times allows managers to influence customer decisions, and the overall system dynamics, which are crucial to a company's productivity and underlying revenue. As a result, many important questions arise when a service decides to implement a delay announcement for its customers. What type of information, if any, should the service provider give to customers? Are there circumstances when the delay announcement hurts the service provider?  How long does it take for the service provider to calculate the delay announcement and disseminate it to customers? 

The existing literature explores different ways to give a delay announcement, as well as different response behaviors of the customers.  For example, Ibrahim et al.  \cite{ibrahim2015does} study a service system with applications to telephone call centers. Upon calling and receiving the delay announcement, customers have the option to join the queue, balk (leave immediately), or abandon the queue after spending some time waiting.  The authors develop methods for determining the accuracy of the last-to-enter-service (LES) delay announcement, which estimates the waiting time for an incoming customer as the waiting time of the most recent customer who entered service.  Under the same options for customer behavior, Jouini et al. \cite{jouini2011call} consider providing different percentiles of the waiting time distribution as information to their customers. They determine the amount of information that maximizes the number of customers who end up receiving service. Armony and Maglaras \cite{armony2004customer} model a queue where upon arrival the customers are told the steady state expected waiting time, and in addition are given an option to request a call back.  The authors propose a staffing rule that picks the minimum number of service agents that satisfies a set of operational constraints on the performance of the system. Guo and Zipkin \cite{guo2007analysis, guo2009impacts} allow customers to either join the queue or balk, when the customers are presented either with no information, partial information, or full information about the queues are disclosed. The authors discuss for what situation the extra information is beneficial, and when the addition information can hurt the customers or the service provider.

This paper also explores the impact of the delay announcement on the dynamics of the queueing process.  However, the current literature focuses only on services that give the delay announcements to their customers in real-time, while we consider the scenarios when the information itself is delayed. Lags in information are common in services that inform their customers about the waiting times prior to customers' arrival to the service. Such services are prominent in the context of hospital emergency rooms, highway transportation, amusement park rides, and internet buffer sizing \cite{dong2018, samantha2018, buffer1, transport2009}. One specific example is the Citibike bike-sharing network in New York City \cite{freund2016bikes,tao2017bikes}. Riders can search the availability of bikes on a smartphone app, as shown on Figure \eqref{Fig: bike}.  However, in the time that it takes for the riders to leave their home and get to a station, all of the bikes could have been taken from that station.  Thus, the information they used is delayed and is somewhat unreliable by the time they arrive to the station. 
\begin{figure}[H]
    \centering
    \includegraphics{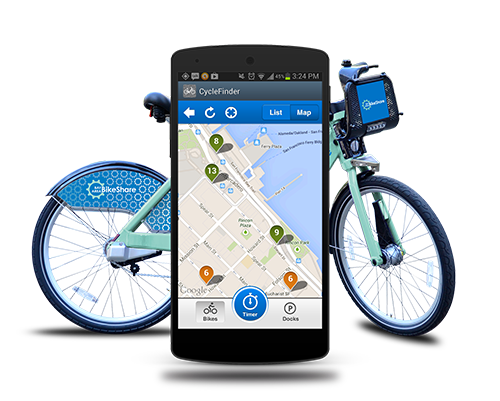}
    \caption{Bike sharing network app.}
    \label{Fig: bike}
\end{figure}

In this paper, we present a deterministic fluid-like model of queues. This may seem as a counter-intuitive choice given that queues usually comprise of discrete units such as the number of people, jobs, or automobiles. However, a queueing system with a heavy traffic flow can be well approximated by a fluid model. These approximations are common in queueing literature \cite{ibrahim2015does,armony2009impact}, where the queues are modeled as stochastic processes and then shown to converge in some limiting regime to deterministic equations. Fluid models are especially useful in settings like manufacturing systems, traffic networks, cloud-computing jobs, busy call centers, and crowds at Disneyland park where the demand for service is large \cite{perkins1995, helbing1995, armony2009impact,armbruster2004, samantha2018}. 

\subsection{Contributions of Paper} 
We present a fluid model of $N$ queues where customers choose which queue to join, giving preference to the shorter queue based on delayed information. Similar models were previously considered by the authors in \cite{dchoice2016,pender2017asymptotic,novitzky2018}. The size of delay in information determines whether the queues approach a stable equilibrium or Hopf bifurcations occur and the queues oscillate indefinitely. The threshold at which the queues become unstable can be affected by the type of information that is revealed to customers. This paper analyzes what kind of information the service managers should provide to customers in order to distribute the workload evenly among the queues. This benefits both the customers who will avoid excessive waits at the longer queues and the servers who will avoid getting overworked or underworked. In many settings, the operator knows not only the current queue lengths, but also the rate at which the queues are changing, namely the queue velocity. 
\begin{itemize}
    \item We develop a new queueing model, where customers are told a weighted sum of the queue length and the queue velocity. 
    \item We show that the queueing system can undergo a Hopf bifurcation if the delay due to the customer travel time is large. We derive the exact point where the Hopf bifurcation occurs.
    \item  We specify how the weight coefficient of velocity information should be chosen so that queues can maintain their stability under greater lags in time. Specifically, we prove that there exists an optimal weight that maximizes the delay where bifurcation occurs. 
    \item We derive a fixed point equation for the optimal weight, as well as closed-form expressions for upper and lower bounds on that weight. We also provide upper and lower bounds on the maximum delay where the bifurcation occurs.
    \item When the oscillations in queues cannot be prevented, we use the second-order approximation of amplitude via Lindstedt's method to determine the weight of velocity information that minimizes the amplitude of oscillations. 
    \item  When the weights are chosen inadequately, the velocity information can be harmful to the system. We specify the threshold for the weight coefficient where the adverse effects take place.

\end{itemize}

\subsection{Organization of Paper}  The remainder of the paper is organized as follows. Section \eqref{Sec: model} presents a mathematical model for $N$ queues and describes the qualitative behavior of the queueing system.  In particular, we prove the existence and uniqueness of the equilibrium and give conditions under which the equilibrium is locally stable. We show that for certain values of parameters, infinitely many Hopf bifurcations may occur.

For some parameters, the queues converge to an equilibrium for sufficiently small delay in information, but as the delay exceeds a certain threshold $\Delta_{cr}$, the equilibrium becomes unstable. Section \eqref{Sec: comparison} discusses how the velocity information affects $\Delta_{cr}$, and since the queues are stable only when the delay is less than $\Delta_{cr}$, it becomes our objective to maximize the threshold delay to provide. 

Section \eqref{Sec: N=2} considers a queueing system with two queues, which is a special case of our $N$-queue model. We prove that all Hopf bifurcations are supercritical. We use a perturbations technique to develop a highly accurate approximation of the amplitude near the bifurcation point, and show that the amplitude of oscillations in queues can be decreased with the right choice of the velocity information weight parameter.

\section{The Queueing Model} \label{Sec: model}
Customers arrive at a rate $\lambda>0$ to a system of $N$ queues, where they are given information about the waiting times at each queue based on the current queue length and the rate at which the queue is changing. Each customer chooses one of $N$ queues to join, giving probabilistic preference to the shorter queue. Then, customers travel for $\Delta>0$ time units to reach the queue of their choice. The model assumes infinite-server queues with service rate $\mu>0$, which is customary in the operations research literature \cite{fralix2009,iglehart1965,resnick1999}. This assumption implies that the departure rate for a queue is the service rate $\mu$ multiplied by the total number of customers in that queue. Figure \eqref{Fig:QueueFlow} shows the queueing system for $N$ queues. The queue length for each queue $i$ is given by 
\begin{eqnarray}
\label{q_i dot with p}
\shortdot{q}_i(t) =  \lambda  p_i(q_1, \dots, q_N, \shortdot{q}_1, \dots,  \shortdot{q}_N, \Delta)  - \mu q_i(t), \quad \forall i \in \{1,...,N\},
\end{eqnarray}
where the function $p_i$ represents the probability that a customer chooses the $i^{th}$ queue.
\begin{figure}[H]
	\centering 
		\includegraphics[scale=.8]{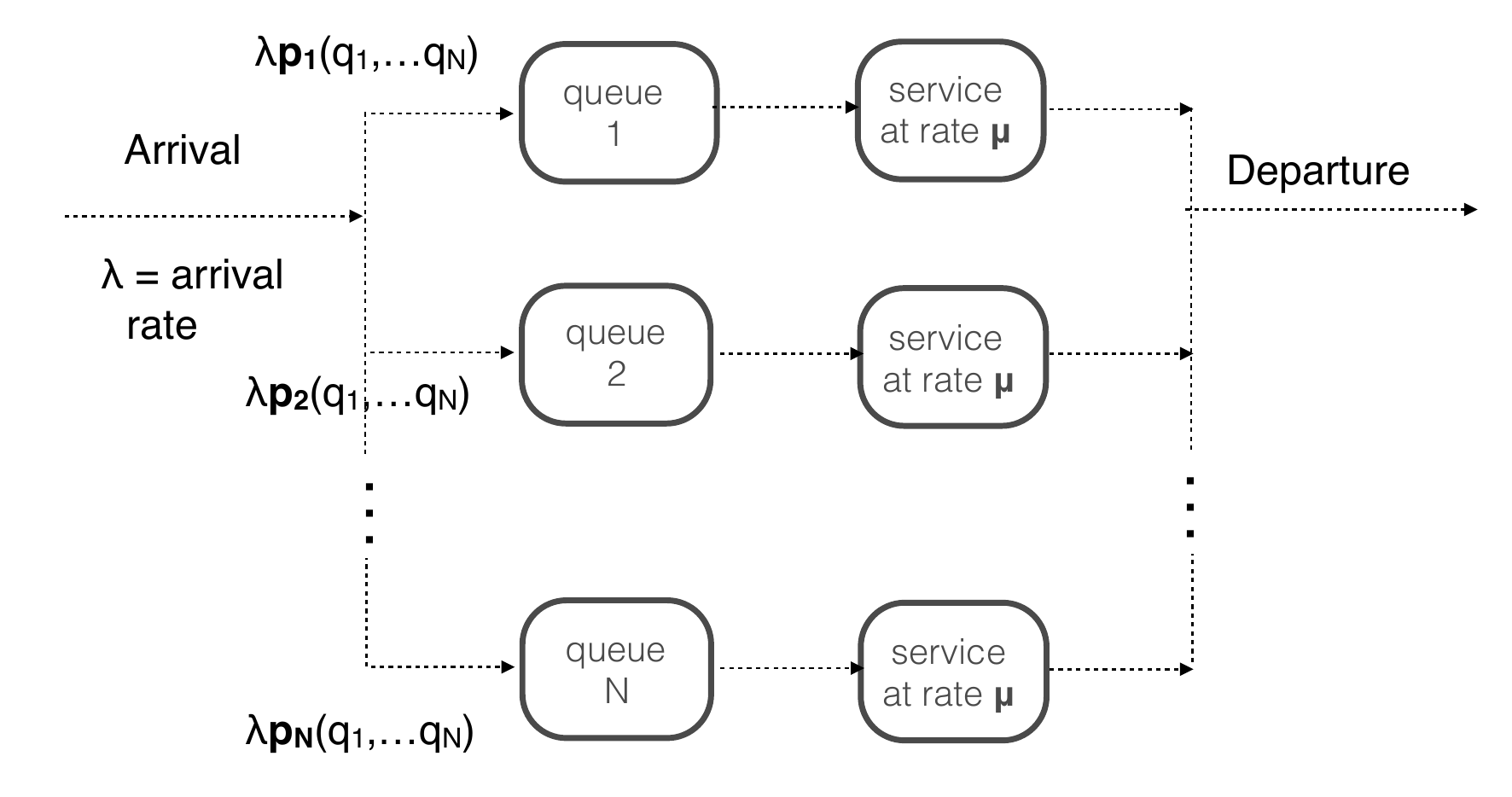}
\caption{Customers going through a N-queue service system.} \label{Fig:QueueFlow}
\end{figure}

Customers are told a weighted sum of the queue length and the queue velocity,
\begin{eqnarray}
\text{information about i}^{th}\text{ queue} = q_i(t- \Delta) + \delta \shortdot{q}_i(t - \Delta).
\end{eqnarray}
This information helps customers decide which queue to join. The probability of a customer choosing the $i$-th queue is given by the Multinomial Logit Model (MNL), commonly used to model customer choice in fields of operations research, economics, and applied psychology \cite{so1995,hausman1984,mcfadden1977,train2009}
\begin{eqnarray}
p_i(q_1, \dots, q_N, \shortdot{q}_1, \dots,  \shortdot{q}_N, \Delta) = \frac{\exp\Big(-\theta\big(q_i(t - \Delta)+ \delta \shortdot{q}_i(t - \Delta)\big)\Big)}{\sum_{j = 1}^N \exp\Big(-\theta\big(q_j(t - \Delta) + \delta \shortdot{q}_j(t - \Delta)\big)\Big)}, \label{p_i}
\end{eqnarray}
where $\theta > 0$ is a standard coefficient of the MNL,  $\delta \geq 0$ is the weight of the information about queue's velocity, and $\Delta >0$ is the delay in time due to customers travelling to the service. 

The parameter $\theta$ determines how strong the customer preference is for the shortest queue. For intuition, we will illustrate the MNL model on the simplest model where there are two queues, and the parameters $\delta, \Delta$ are set to $0$. Figure \eqref{Fig: MNL surf} shows the probability of a customer joining the 1st queue, as a function of $\theta$ and the difference in queue lengths $q_2 - q_1$, 
\begin{eqnarray}
p_1(q_1,q_2) = \frac{\exp(-\theta q_1)}{\exp(-\theta q_1) +  \exp(-\theta q_2)} = \frac{1}{1 +  \exp(-\theta \big(q_2- q_1)\big)} .
\end{eqnarray}
When $\theta \to 0$, customers choose queues arbitrarily giving no preference based on the queue length. Figure \eqref{Fig: MNL surf} indicates the $\theta = 0$ by the yellow line, and $p_1 = p_2 = 0.5$ for any difference in queue lengths. When $\theta \to \infty$, customers always choose the shortest queue, even when the difference in lengths is marginal. This is marked by the black line in Figure \eqref{Fig: MNL surf}. However, for simplicity one can set $\theta = 1$, which is denoted by the red curve. In this case, when the queues are roughly of equal length, they will be joined with roughly the same probabilities, but ones the difference in the queue lengths increases, the shorter queue will become more preferable. 

\begin{figure}[H]
	\centering 
		\includegraphics[scale=.4]{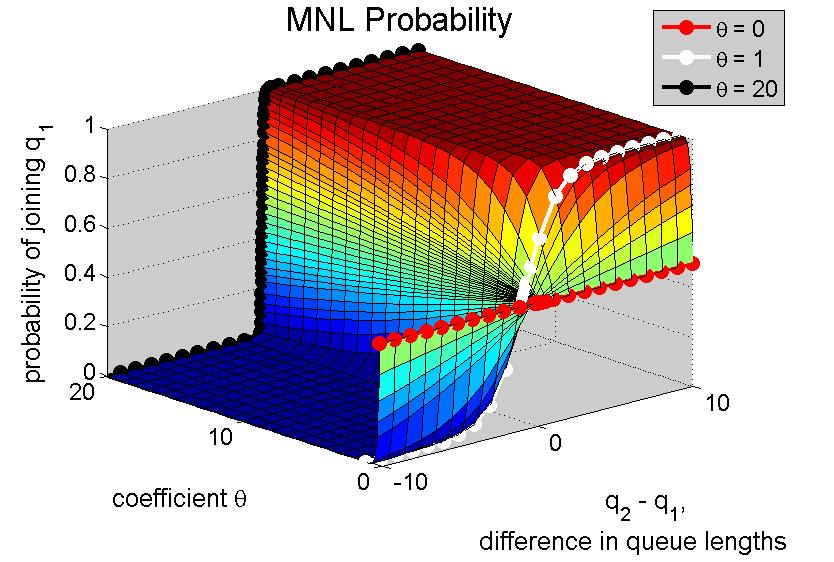}
\caption{MNL for two queues.} \label{Fig: MNL surf}
\end{figure} 

\paragraph{Complete model.} The incorporation of the probabilities $p_i$ into the queueing system provides a system of neutral delay differential equations (NDDE's) for the queue lengths
\begin{eqnarray}
\label{q_i dot}
\shortdot{q}_i(t) =  \lambda \cdot \frac{\exp\Big(-\theta\big(q_i(t - \Delta)+ \delta \shortdot{q}_i(t - \Delta)\big)\Big)}{\sum_{j = 1}^N \exp\Big(-\theta\big(q_j(t - \Delta) + \delta \shortdot{q}_i(t - \Delta)\big)\Big)}  - \mu q_i(t), \quad \forall i \in \{1,...,N\}
\end{eqnarray}
with the initial conditions specified by nonnegative continuous functions $f_i$ 
\begin{eqnarray} \label{initial conditions}
q_i(t) = f_i(t), \quad \shortdot{q}_i(t) = \shortdot{f}_i(t), \quad t \in [-\Delta, 0] .
\end{eqnarray}


\subsection{Conditions for Stability and Hopf Bifurcations}

In this Section, we describe the behavior of the queues from Equation \eqref{q_i dot}. We begin by establishing the existence and uniqueness of solution to the initial value problem \eqref{q_i dot} - \eqref{initial conditions}. We note that there exists an extensive analysis of functional differential equations, see for example \cite{hale1993, Bellena2009, lipshutz2015}. The existence and uniqueness of solution for our specific model directly follows from Driver \cite{Driver1965}, as stated in the result below.
\begin{theorem}
Let $f_i(t)$ from Equation \eqref{initial conditions} be absolutely continuous on $t \in [-\Delta,0]$, and $\shortdot{f}_i(t)$ be bounded for almost all $t \in [-\Delta,0]$ for every $1 \leq i \leq N$. Then there exists a solution $q_i,\dots, q_N$ for all $t>0$ that satisfies Equations \eqref{q_i dot} - \eqref{initial conditions}. Further, the solution is unique.
\end{theorem}
\begin{proof}
The existence of the solution is given by Theorem 1 of Driver \cite{Driver1965}. The uniqueness of the solution follows from Theorem 2 of Driver \cite{Driver1965}, but we first need to ensure that the conditions of Theorem 2 are fulfilled. The theorem requires that the function
\begin{eqnarray}
\lambda \cdot \frac{\exp\Big(-\theta\big(q_i(t - \Delta)+ \delta \shortdot{q}_i(t - \Delta)\big)\Big)}{\sum_{j = 1}^N \exp\Big(-\theta\big(q_j(t - \Delta) + \delta \shortdot{q}_i(t - \Delta)\big)\Big)}  - \mu q_i(t), \quad \forall i \in \{1,...,N\}
\end{eqnarray}
satisfies the local Lipschitz condition with respect to $(q_1(t),\dots, q_N(t))$ with Lipschitz constant $L$ where $L$ is a continuous function of $(\shortdot{q}_1(t - \Delta), \dots,\shortdot{q}_N(t - \Delta))$. Here $q_i(t)$ and $q_i(t - \Delta)$ are treated as different variables, so the local Lipschitz condition with respect to $(q_1(t),\dots, q_N(t))$ is satisfied trivially with $L = 2\mu$. Therefore, the solution to the system \eqref{q_i dot} - \eqref{initial conditions} is guaranteed to be unique.
\end{proof}

\begin{theorem}\label{Thm: equilibrium existence}
The unique equilibrium of $q_i(t)$ from Equation \eqref{q_i dot} is given by 
\begin{eqnarray}
q_i(t) = q_i^* =  \frac{\lambda}{N\mu}, \quad 1 \leq i \leq N. \label{equilibrium}
\end{eqnarray}
\end{theorem}
\begin{proof}
See the Appendix for the proof.
\end{proof}


The stability of the equilibrium can be determined by the stability of the linearized system of equations \cite{hale1993,smith2011}. Hence, we proceed by linearizing $q_i$ about the equilibrium, and 
finding the characteristic equation.


\begin{proposition} \label{Proposition: char eq}
The characteristic equation of \eqref{q_i dot} is given by
\begin{eqnarray}
\Phi(R, \Delta) &=& -R -  \frac{\lambda \theta}{N} \Big( e^{-R\Delta} + \delta R e^{-R\Delta}  \Big)  - \mu  = 0. \label{char eq}
\end{eqnarray}
\end{proposition}
\begin{proof}
We introduce the functions $u_i(t)$ that represent the deviation of $q_i(t)$ from the equilibrium:
\begin{eqnarray}
u_i(t) = q_i(t) - q_i^* =  q_i(t) -  \frac{\lambda}{N \mu} .
\end{eqnarray}
Once the NDDE's are linearized, $\shortdot{u}_i(t)$ are approximated as
\begin{eqnarray}
\shortdot{u}_i(t) \approx -  \frac{\lambda\theta }{N}\Big(u_i(t-\Delta) +  \delta u'_i(t-\Delta) \Big) + \frac{\lambda\theta }{N^2}\sum_{j =1}^N\Big(u_j(t-\Delta) +  \delta u'_j(t-\Delta) \Big)  - \mu u_i(t) 
\end{eqnarray}
In the vector form, we have 
\begin{eqnarray}
\label{u dot vector}
\shortdot{{\bf u}}(t) = -  \frac{\lambda \theta }{N} \cdot \big({\bf u}(t-\Delta) + \delta {\bf u'}(t-\Delta) \big) + \frac{\lambda\theta }{N^2} A \big({\bf u}(t-\Delta) + \delta {\bf u'}(t-\Delta) \big) - \mu {\bf u}(t)  ,
\end{eqnarray}
where $A \in \mathbb{R}^{N \times N}$, and $A_{ij} = 1$ for $1 \leq i,j \leq N$. The matrix $A$ can be diagonalized:
\begin{eqnarray}
A = VDM, \quad \text{where} \quad  V,D,M \in \mathbb{R}^{N \times N} ,\\
VM = MV = I, \quad D_{ij} = 0 \quad \text{if} \quad i \neq j.
\end{eqnarray}
Since all rows of $A$ are identical, $A$ has only one eigenvalue. This implies diagonal matrix $D$  has only one nonzero element $D_{11} = N$.  This property can be exploited with the introduction of a vector ${\bf w}(t)$:
\begin{eqnarray}
{ \bf u}(t) = V{\bf w}(t).
\end{eqnarray}
That is an acceptable form of definition because $V$ is invertible. Equation \eqref{u dot vector} becomes 
\begin{eqnarray}
V\shortdot{{\bf w}}(t) = -  \frac{\lambda \theta}{N}  V\big({\bf w}(t-\Delta) + \delta {\bf w'}(t-\Delta) \big) \\
+ \frac{\lambda \theta}{N^2} VDM V\big({\bf w}(t-\Delta) + \delta {\bf w'}(t-\Delta) \big)  - \mu V{\bf w}(t).
\end{eqnarray}
Pre-multiplying this equation by $M$ yields the following simplification,
\begin{eqnarray}
\shortdot{{\bf w}}(t) = -  \frac{\lambda \theta}{N}  \big({\bf w}(t-\Delta) + \delta {\bf w'}(t-\Delta) \big) \\
+ \frac{\lambda\theta}{N^2} D\big({\bf w}(t-\Delta) + \delta {\bf w'}(t-\Delta) \big) - \mu {\bf w}(t).
\end{eqnarray}
Writing out $D$ explicitly reduces the system of $N$ equations down to just two equations:
\begin{eqnarray}
\label{w 1 dot}
\shortdot{w}_1(t) &=&  - \mu w_1(t), \\
\label{w 2 dot}
\shortdot{w}_i(t) &=&  -  \frac{\lambda \theta}{N} \big( w_i(t-\Delta)+\delta \shortdot{w}_i(t-\Delta)\big)  - \mu w_i(t), \qquad i \neq 1 .
\end{eqnarray}
Equation \eqref{w 1 dot} has a solution of the form $w_1(t) = a e^{- \mu t}$, so $w_1 \to 0$ over time. By assuming a solution of the form $w_i(t) = e^{Rt}$, the characteristic equation then follows from \eqref{w 2 dot}.
\end{proof}


The equilibrium is stable when all eigenvalues $R$ of the characteristic equation have negative real parts. It is evident that any real root $R$ must be negative. However, there are also infinitely many complex roots, and they depend on the delay $\Delta$. When $\delta > \frac{N}{\lambda\theta}$, the equilibrium cannot be stable for any $\Delta>0$ because there are infinitely many eigenvalues with positive real parts. We demonstrate this in the result below.

\begin{proposition}\label{proposition: positive eigs}
Suppose $\delta > \frac{N}{\lambda\theta} $. Then for any $\Delta>0$, there are infinitely many eigenvalues of the characteristic equation that have positive real parts. 
\end{proposition}
\begin{proof}
Suppose $\delta > \frac{N}{\lambda\theta} $. We assume $R = a+ i b$ with $a,b \in \mathbb{R}$. We can assume $b \geq 0$ without loss of generality. Plugging in $R$ and separating the real and imaginary parts:
\begin{eqnarray}
-(a+\mu)N = e^{-a\Delta} \lambda\theta \Big((1+a\delta)\cos(b\Delta) + b\delta \sin(b \Delta)\Big) \\
bN = e^{-a\Delta} \lambda\theta\Big(-b\delta  \cos(b\Delta) + (1+a\delta)  \sin(b \Delta)\Big).
\end{eqnarray}
We find the expressions for sine and cosine to be
\begin{eqnarray}
\cos(b\Delta) &=& - \frac{e^{a\Delta}\Big((a + \mu)(1+a\delta)N + b^2\delta N\Big)}{\lambda\theta \Big((a\delta + 1)^2 + (b\delta)^2\Big)} \label{cos complex 1}\\ \label{sin complex 1}
\sin(b\Delta) &=& -\frac{e^{a\Delta}Nb(\delta\mu - 1)}{\lambda \theta \big( (1+a\delta)^2 + (b\delta)^2 \big)}.
\end{eqnarray}
The identity $\sin^2(b\Delta) + \cos^2(b\Delta) = 1$ gives an expression for $b$, 
\begin{eqnarray}
b = \sqrt{\frac{\lambda^2\theta^2(a\delta + 1)^2 - e^{2a\Delta}N^2(a+\mu)^2}{e^{2a\Delta}N^2 - \delta^2\lambda^2\theta^2}}. \label{b complex 1}
\end{eqnarray}
We will now show that there are infinitely many eigenvalues $R$, where Re$[R] = a >0$, by separately considering the cases when $\delta\mu >1$, $\delta \mu <1$, and $\delta \mu = 1$. 

\paragraph{Case 1: $\delta\mu >1$.} We will construct an interval $(a_1,a_2)$ with $0<a_1<a_2$, which contains infinitely many values Re$[R] = a$ that together with $b$ from Equation \eqref{b complex 1} satisfy the characteristic equation. We will choose $a$ to be such that both the numerator and the denominator of $b$ are negative, therefore guaranteeing $b$ to be be real. This yields two inequalities
\begin{eqnarray}
\frac{\delta \theta \lambda}{N} > e^{a\Delta} > \frac{\theta \lambda (1+a\delta)}{N(a+\mu)}. \label{inequality 1}
\end{eqnarray}
Since $\frac{\delta \theta \lambda}{N} >1$ by the assumption that $\delta >\frac{N}{\theta \lambda}$, then the inequality $\frac{\delta \theta \lambda}{N} > e^{a\Delta}$ holds for $a \in [0, a_2)$, where $a_2 = \frac{1}{\Delta} \ln(\frac{\delta \theta \lambda}{N}) >0$. Further, as $a$ increases, the exponent $e^{a\Delta}$ must inevitably outgrow $\frac{\theta \lambda (1+a\delta)}{N(a+\mu)}$, so there exists $a_1 \geq 0$ such the second part of inequality from Equation \eqref{inequality 1} holds for all $a \geq a_1$. Lastly, note that the condition $\frac{\delta \theta \lambda}{N}>\frac{\theta \lambda (1+a\delta)}{N(a+\mu)}$ holds for all $a \geq 0$ because $\delta\mu >1$, so we can choose $a_1$ to be less than $a_2$, i.e. $a_1 \in (0, a_2)$. This shows that there exists an interval $(a_1, a_2)$ with $0 < a_1 < a_2$ where the inequalities from Equation \eqref{inequality 1} hold, so by Equation \eqref{b complex 1} we have $0 \neq b \in R$ for all $a \in (a_1, a_2)$.

If $b \in \mathbb{R}$ satisfies Equation \eqref{sin complex 1} for some value of $a$, then $R$ is an eigenvalue of the characteristic equation.  To show that there are infinitely many eigenvalues with real parts in $(a_1,a_2)$, we consider the limit $a \to a_2^-$, when the denominator of $b$ approaches zero and $b \to \infty$. In this limit, the right hand side of Equation \eqref{sin complex 1} will oscillate between $-1$ and $1$ infinite number of times, while the left hand side of Equation \eqref{sin complex 1}  will converge to $0$. Hence, there are infinitely many solutions to Equation \eqref{sin complex 1} with $a\in (a_1, a_2)$, and so there are infinitely many eigenvalues with positive real parts. 

\paragraph{Case 2: $\delta\mu <1$.} The argument here is analogous to Case 1, except to guarantee that $b$ from Equation \eqref{b complex 1} is real-valued, we will determine an interval in the range of $a$ where the numerator and the denominator of $b$ are {\it positive}. We get the condition 
\begin{eqnarray}
\frac{\delta \theta \lambda}{N} < e^{a\Delta} < \frac{\theta \lambda (1+a\delta)}{N(a+\mu)}. \label{inequality 2}
\end{eqnarray}
 At $a = 0$, $\frac{\theta \lambda (1+a\delta)}{N(a+\mu)} = \frac{\theta \lambda\delta}{N \mu \delta} > \frac{\delta\theta \lambda}{N } > 1$, so there is an interval $[0, a_2)$ for $a$ where $e^{a\Delta} < \frac{\theta \lambda (1+a\delta)}{N(a+\mu)}$ holds. Further, $\frac{\delta \theta \lambda}{N} < \frac{\theta \lambda (1+a\delta)}{N(a+\mu)}$ holds for all $a \geq 0$ because $\delta\mu <1$, therefore $a_1 = \frac{1}{\Delta} \ln(\frac{\delta \theta \lambda}{N})>0$ must be smaller than $a_2$. Therefore for all $a \in (a_1, a_2)$, with $0< a_1 <a_2$, $b \in \mathbb{R}$.

Just as in Case 1, when $a \to a_1^+$, $b \to \infty$ so the right hand side of Equation \eqref{sin complex 1} will oscillate between 1 and -1 infinitely many times, while the left hand side will converge to $0$. Thus, there will be infinitely many eigenvalues that satisfy the characteristic equation \eqref{char eq}. 

\paragraph{Case 3: $\delta\mu =1$.} In this case, the expressions for sine and cosine simplify to
\begin{eqnarray}
\cos(b \Delta) = -\frac{e^{a\Delta}N}{\lambda\theta\delta}, \quad \sin(b\Delta) = 0,
\end{eqnarray}
so $b = (2k-1)\pi/\Delta$ for $k = 1,2, \dots $, and $a = \frac{1}{\Delta}\ln(\frac{\lambda\theta\delta}{N})$. Since $\delta > \frac{N}{\theta \lambda}$, then $a>0$, and the characteristic equation\eqref{char eq} has infinitely many eigenvalues with positive real parts. 
\end{proof}


However, when the weight coefficient $\delta$ is sufficiently small, i.e. $\delta < \frac{N}{\lambda\theta}$, then given a sufficiently small delay the queues converge to a stable equilibrium. As the next result shows, the stability is due to all complex eigenvalues having negative real parts.


\begin{proposition} \label{proposition: negative eigs}
Suppose $\delta < \frac{N}{\lambda\theta} $. When $\Delta$ is sufficiently small, all eigenvalues of the characteristic equation have negative real parts. 
\end{proposition}
\begin{proof}
To reach contradiction, let us assume that for any $\Delta_0>0$ there exists some $\Delta \in (0, \Delta_0)$ and an eigenvalue $R = a+i b$ with $a \geq 0$ that satisfy the characteristic equation \eqref{char eq}. We can assume $b \geq 0$ without loss of generality. Plugging in $R$ and separating the real and imaginary parts: 
\begin{eqnarray}
-(a+\mu)N = e^{-a\Delta} \lambda\theta \Big((1+a\delta)\cos(b\Delta) + b\delta \sin(b \Delta)\Big) \\
bN = e^{-a\Delta} \lambda\theta\Big(-b\delta  \cos(b\Delta) + (1+a\delta)  \sin(b \Delta)\Big).
\end{eqnarray}
Solving for sine and cosine, we find 
\begin{eqnarray}
\cos(b\Delta) &=& - \frac{e^{a\Delta}\Big((a + \mu)(1+a\delta)N + b^2\delta N\Big)}{\lambda\theta \Big((a\delta + 1)^2 + (b\delta)^2\Big)} \label{cos complex}\\ \label{sin complex}
\sin(b\Delta) &=& -\frac{e^{a\Delta}Nb(\delta\mu - 1)}{\lambda \theta \big( (1+a\delta)^2 + (b\delta)^2 \big)} 
\end{eqnarray}
The identity $\sin^2(b\Delta) + \cos^2(b\Delta) = 1$ gives an expression for $b$, 
\begin{eqnarray}
b = \sqrt{\frac{\lambda^2\theta^2(a\delta + 1)^2 - e^{2a\Delta}N^2(a+\mu)^2}{e^{2a\Delta}N^2 - \delta^2\lambda^2\theta^2}}. \label{b complex}
\end{eqnarray}
Since $e^{2a\Delta} \geq 1$ and $N>\delta\lambda\theta$ by assumption, the denominator of $b$ is positive, so the numerator of $b$ must be non-negative. Therefore we get inequalities
\begin{eqnarray}
1 \leq e^{a\Delta} \leq \frac{\lambda\theta(a\delta + 1)}{N(a+\mu)}, \quad e^{a\Delta}> \frac{\delta \lambda\theta}{N}. \label{ineq 1}
\end{eqnarray}
From the first inequality, we obtain an upper bound on $a$, $a \leq \frac{\lambda\theta - N\mu}{N - \lambda\theta \delta}$. If $\lambda\theta < N\mu$, then $a < 0$ so we reached a contradiction. If $\lambda\theta \geq N\mu$, then we use \eqref{b complex} and \eqref{ineq 1} to find an upper bound on $b$:
\begin{eqnarray}
b \leq B = \sqrt{\frac{2N\lambda^2\theta^2(1 - \delta \mu)^2}{N^2 - \delta^2 \lambda^2\theta^2}}, \quad B>0. \label{bound B}
\end{eqnarray}
However, we note from the cosine equation \eqref{cos complex} that $\cos(b\Delta)<0$. Since $b$ is non-negative then $b\Delta > \frac{\pi}{2}$ so $b > \frac{\pi}{2\Delta}$ for any $\Delta$. Choose $\Delta_0 = \frac{\pi}{4B}$. Then for any $\Delta< \Delta_0$ we get a contradiction with \eqref{bound B}:
\begin{eqnarray}
b > \frac{\pi}{2\Delta} > \frac{\pi}{2\Delta_0} > 2B.
\end{eqnarray}
Hence when $\Delta_0$ is sufficiently small, then for any $\Delta < \Delta_0$ the real part of any eigenvalue is negative. 
\end{proof}


An interesting edge case, however, is when $\delta = \frac{N}{\lambda \theta}$. If the equality holds, three different behaviors may be observed. The equilibrium will be stable, regardless of the size of the delay, if $\delta \mu > 1$. However, the equilibrium will be unstable if $\delta \mu <1$. Further, if $\delta \mu = 1$ then the behavior of the queues cannot be determined from the characteristic equation \eqref{char eq}, as all eigenvalues will be purely imaginary. We justify these findings in the result below. 


\begin{proposition}\label{prop: delta = N/lambda theta}
Suppose $\delta = \frac{N}{\lambda \theta}$. If $\delta \mu < 1$, then for any $\Delta$ there exists at least one eigenvalue with positive real part. If $\delta \mu > 1$, then all eigenvalues have negative real parts. Further, if $\delta \mu = 1$ then all eigenvalues are purely imaginary.
\end{proposition}
\begin{proof}
As in Propositions \eqref{proposition: positive eigs} - \eqref{proposition: negative eigs}, we express the eigenvalue as $R = a + i b$, and then separate the real and imaginary parts of the characteristic equation. The assumption $\delta = \frac{N}{\lambda \theta}$ simplifies the expressions to be
\begin{eqnarray}
\sin(b \Delta) = -\frac{b N e^{a \Delta } (\mu  N-\theta  \lambda )}{N^2b^2 + (a N +\theta \lambda)^2} \label{sin expr} \\
\cos(b \Delta) = -\frac{N e^{a \Delta } \left(a^2 N+a \theta  \lambda +a \mu  N+b^2 N+\theta  \lambda  \mu \right)}{N^2b^2 + (a N +\theta \lambda)^2} \label{cos expr}
\end{eqnarray}
We will address the three cases separately. 
\paragraph{Case 1: $\delta\mu >1$.} To reach contradiction, suppose there exists an eigenvalue with a nonnegative real part, $a \geq 0$. The expression for $b$ is given by 
\begin{eqnarray}
b = \frac{\sqrt{(a N+\theta  \lambda )^2-N^2 e^{2 a \Delta } (a+\mu )^2}}{\sqrt{N^2 \left(e^{2 a \Delta }-1\right)}} \label{b expr}
\end{eqnarray}
where the denominator is positive, so the numerator must be nonnegative for $b$ to be real. Therefore $ a N + \theta  \lambda - N e^{ a \Delta } (a+\mu ) >0$. However, the assumption $\delta\mu >1$ is equivalent to $\lambda \theta < N \mu$, so we can show that
\begin{eqnarray}
a N + \theta  \lambda - N e^{ a \Delta } (a+\mu ) \leq a N + \theta  \lambda - N (a+\mu ) = \theta \lambda - N \mu <0, 
\end{eqnarray} 
and we reached a contradiction. Thus, if $\delta\mu >1$ then any eigenvalue must have a negative real part.
\paragraph{Case 2: $\delta\mu < 1$.} This condition is equivalent to  $\lambda \theta < N \mu$. Again, $b$ satisfies Equation \eqref{b expr}. As $a \to 0^+$, $b \to \infty$ so $\sin(b \Delta)$ oscillates between 1 and $-1$ infinitely quickly. Further, as $a \to 0^+$ the right hand side of Equation \eqref{sin expr}  goes to zero. Therefore, Equation \eqref{sin expr} will have infinitely many roots, while Equation \eqref{cos expr} will be satisfied at each root automatically since $b$ is given by Equation \eqref{b expr}. Therefore $\delta\mu < 1$ implies that the characteristic equation will have infinitely many eigenvalues with positive real parts. 
\paragraph{Case 3: $\delta\mu =1$.} This case is equivalent to the condition $\lambda \theta = N \mu$, which simplifies the Equations \eqref{sin expr} - \eqref{cos expr} to be
\begin{eqnarray}
\sin(b \Delta) = 0, \qquad \cos(b \Delta) = -e^{{a \Delta}}.
\end{eqnarray}
Hence, $b = (2k+1) \pi/ \Delta$, where $k = 0, 1, 2, \dots$, and $1 = e^{{a \Delta}}$ so $a = 0$. Therefore the roots of the characteristic equation \eqref{char eq} are purely imaginary.
\end{proof}

Hence, the equilibrium is stable when $\delta = \frac{N}{\lambda\theta}$ and $\delta \mu > 1$, or when $\delta < \frac{N}{\lambda\theta}$ and the delay $\Delta$ is sufficiently small. Further, the only way for the equilibrium to become unstable given that $\delta < \frac{N}{\lambda\theta}$, is if a pair of complex eigenvalues crosses from the negative real side of the complex plane into the positive real side. We will determine the threshold value of delay where the stability of the equilibrium may change by finding where the eigenvalues (if any) on the complex plane reach the imaginary axis. 


\begin{proposition} \label{proposition: im roots}
The characteristic equation \eqref{char eq} has a pair of purely imaginary solutions $R = \pm i \omega_{cr}$ with $\omega_{cr}$ being real and positive, at each root $\Delta_{cr}$, given that
\begin{eqnarray} \label{omega cr}
\omega_{cr} = \sqrt{\frac{\lambda^2\theta^2 - N^2\mu^2}{N^2 - \delta^2\lambda^2\theta^2}} 
\end{eqnarray}
and $\Delta_{cr}$ satisfies the transcendental equation
\begin{eqnarray}
\cos \bigg(\Delta_{cr}\sqrt{\frac{\lambda^2\theta^2 - N^2\mu^2}{N^2 - \delta^2\lambda^2\theta^2}}\bigg)  = - \frac{\delta \lambda^2\theta^2 + N^2 \mu}{N \lambda\theta(1 +  \delta \mu)}. \label{Hopf condition}
\end{eqnarray}
\end{proposition}
\begin{proof}
Assume that $R$ from the characteristic equation \eqref{char eq} is purely imaginary, $R = \pm i \omega_{cr}$. Plugging in $R$, the real and imaginary parts produce two equations
\begin{eqnarray}
\mu = -\frac{\lambda\theta}{N}\cos(\omega_{cr}\Delta_{cr}) - \frac{\lambda\theta}{N}\delta \omega_{cr} \sin(\omega_{cr}\Delta_{cr}) \\
\omega_{cr} = \frac{\lambda\theta}{N}\sin(\omega_{cr}\Delta_{cr}) - \frac{\lambda\theta}{N}\delta \omega_{cr} \cos(\omega_{cr}\Delta_{cr})
\end{eqnarray}
We can solve for the values of the sine and cosine functions i.e.
\begin{eqnarray}
\cos(\omega_{cr}\Delta_{cr})  = - \frac{N(\mu + \delta \omega_{cr}^2)}{\lambda\theta(1 + \delta^2 \omega_{cr}^2)}, \quad \sin(\omega_{cr}\Delta_{cr}) = \frac{N \omega_{cr}(1 - \delta \mu)}{\lambda\theta(1  + \delta^2 \omega_{cr}^2)} \label{cos Hopf}
\end{eqnarray}
and by the trigonometric identity $\sin^2(\omega_{cr}\Delta_{cr}) +\cos^2(\omega_{cr}\Delta_{cr}) =1$, $\omega_{cr}$ is found. The cosine equation from \eqref{cos Hopf} then gives the equation for $\Delta_{cr}$.
\end{proof}


Proposition \eqref{proposition: im roots} provides the infinitely many critical delays $\Delta_{cr}$ as well as the necessary conditions on the other parameters ($\omega_{cr} \in \mathbb{R}$, $\omega_{cr} \neq 0$) for when Hopf bifurcations may occur. This information allows us to prove that a Hopf bifurcation occurs at every $\Delta_{cr}$. 

\begin{theorem}\label{thm: Hopf}
Suppose $\omega_{cr}$ from Equation \eqref{omega cr} is real and nonzero. Then a Hopf bifurcation occurs at $\Delta = \Delta_{cr}$, where $\Delta_{cr}$ is any positive root of
\begin{equation} \label{delta cr}
\Delta_{cr}(\lambda, \mu) = \arccos \bigg(- \frac{\delta \lambda^2\theta^2 + N^2 \mu}{N \lambda\theta(1 + \delta \mu)}\bigg) \cdot \sqrt{\frac{N^2 - \delta^2\lambda^2\theta^2}{\lambda^2\theta^2 - N^2\mu^2}}.
\end{equation}
\end{theorem}
\begin{proof}
By Proposition \eqref{proposition: im roots}, at each $\Delta_{cr}$ there is a pair of purely imaginary eigenvalues $R = i \omega_{cr}$, $\bar{R} = -i \omega_{cr}$. A Hopf bifurcation can only occur if $\frac{d}{d\Delta}$Re$[R(\Delta_{cr})] \neq 0$. To verify this, we assume that $R(\Delta) = \alpha(\Delta) + i \omega(\Delta)$. The characteristic equation \eqref{char eq} is differentiated with respect to delay, and we find that at $\Delta_{cr}$ where $\alpha = 0$ and $\omega = \omega_{cr}$, $\frac{d}{d\Delta}$Re$[R]$ is given by
\begin{eqnarray} \label{dalpha ddelta}
\frac{d\alpha}{d \Delta} = \frac{(N^2 - \delta^2 \lambda^2\theta^2)(1 + \delta^2 \omega^2)\omega^2}{\lambda^2\theta^2(1 + \delta^2 \omega^2)\big((\delta - \Delta)^2 + \delta^2 \Delta^2 \omega^2 \big) + N^2\big(1 - 2\delta\mu + 2\Delta\mu +\delta^2\omega^2(2\Delta\mu -1)\big)}.
\end{eqnarray}
The assumption $\omega_{cr} >0$ guarantees the numerator of $\frac{d\alpha}{d \Delta}(\Delta_{cr})$ to be nonzero. To show that the denominator $D$ is nonzero as well, note that it is a quadratic function of $\Delta$, with an absolute minimum at $\Delta^*$ such that  
\begin{eqnarray} \label{delta *}
\frac{d D}{d\Delta}(\Delta^*) = 0 \implies \Delta^* = \frac{\delta \lambda^2\theta^2 - N^2 \mu}{\lambda^2\theta^2(1+\delta^2\omega_{cr}^2)}.
\end{eqnarray}
Once  $\Delta^*$ and $\omega = \omega_{cr}$ from Equations \eqref{delta *} and \eqref{omega cr} are substituted into the denominator $D(\Delta)$ from Equation \eqref{dalpha ddelta}, we find that the minimum of $D$ with respect to $\Delta$ is positive:
\begin{eqnarray}  
D(\Delta) \geq D(\Delta^*) = \frac{(N^2 - \delta^2\lambda^2\theta^2)(\lambda^2\theta^2 - N^2\mu^2)}{\lambda^2\theta^2} = \frac{(N^2 - \delta^2\lambda^2\theta^2)^2\omega_{cr}}{\lambda^2\theta^2} >0,
\end{eqnarray}
Hence the denominator of $\frac{d\alpha}{d \Delta}(\Delta_{cr})$ is positive for any delay $\Delta$, so 
\begin{eqnarray}
\frac{d\alpha}{d \Delta}(\Delta_{cr}) \neq 0.
\end{eqnarray}
In fact, if $\delta < \frac{N}{\lambda\theta}$ then $\frac{d\alpha}{d \Delta}(\Delta_{cr}) > 0$ so the eigenvalues always cross from left to right on the complex plane. If $\delta > \frac{N}{\lambda\theta}$ then $\frac{d\alpha}{d \Delta}(\Delta_{cr}) < 0$ so the eigenvalues always cross from right to left. At each root of $\Delta_{cr}$ there is one purely imaginary pair of eigenvalues, but all other eigenvalues necessarily have a nonzero real part. Hence all roots $\Lambda_j \neq R, \bar{R}$ satisfy $\Lambda_j \neq m R, m\bar{R}$ for any integer $m$. Therefore all conditions of the infinite-dimensional version of the Hopf Theorem from Hale and Lunel \cite{hale1993} are satisfied, so a Hopf bifurcation occurs at every root $\Delta_{cr}$.
\end{proof}

In the proof of Theorem \eqref{thm: Hopf}, for $\delta < \frac{N}{\lambda \theta}$ it is shown that any pair of complex eigenvalues, which crosses the imaginary axis on the complex plane, necessarily crosses from left to right. The implication here is that once the real part of an eigenvalue becomes positive, it remains positive as the delay increases. This allows us to state the conditions for the local stability of the equilibrium.


\begin{theorem}\label{Thm: equilibrium stability}
When $\lambda\theta > N\mu$ and $\delta <\frac{N}{\lambda\theta}$, the equilibrium is locally stable for sufficiently small delay $\Delta$. When either $\lambda\theta \leq N\mu$ and $\delta <\frac{N}{\lambda\theta}$ or $\lambda\theta < N\mu$ and $\delta =\frac{N}{\lambda\theta}$, the equilibrium is locally stable for all $\Delta$.
\end{theorem}
\begin{proof}
If $\delta =\frac{N}{\lambda\theta} $ and $\lambda\theta < N\mu$ then by Proposition \eqref{prop: delta = N/lambda theta}, for any delay all eigenvalues of the characteristic equation have negative real parts, therefore the equilibrium is locally stable. 

If $\delta <\frac{N}{\lambda\theta} $, then by Proposition \eqref{proposition: negative eigs} there exists a sufficiently small $\Delta$ such that all eigenvalues of the characteristic equation have negative real parts. The only way for the equilibrium to become unstable is for an eigenvalue to reach the imaginary axis for some $\Delta$. For that to happen, $\omega_{cr} = \sqrt{\frac{\lambda^2\theta^2 - N^2\mu^2}{N^2 - \delta^2\lambda^2\theta^2}} \in \mathbb{R}$, $\omega_{cr} \neq 0$ must hold. In case when $\lambda\theta \leq N\mu$ and $\delta < \frac{N}{\lambda\theta}$, then either $\omega_{cr} \notin \mathbb{R}$ or $\omega_{cr} = 0$ so the eigenvalues have negative real parts for all $\Delta$. Therefore, again the eigenvalues have negative real parts for all (finite) $\Delta$. Finally,
\end{proof}


%
%
%
%
%
%
%
%
%
%
%
%
%
%
%
%
To summarize, the behavior of the queues from Equation \eqref{q_i dot} can be categorized by two cases, when $\lambda\theta < N\mu$ and $\lambda\theta > N\mu$. In each case, two different types of behavior can be observed, depending on the size of the parameter $\delta$. Hence, there can be four qualitatively different scenarios, as shown in Figure \eqref{Fig: queue stability cases 0}. In the following discussion of the two cases, we will refer to this diagram and will explain it in detail. 

  \begin{figure}[H]
	\centering 
		\includegraphics[scale=.5]{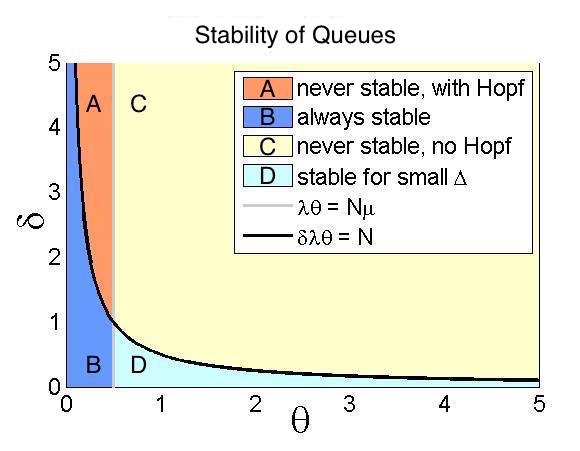}
    \caption{The four stability cases.}
    \label{Fig: queue stability cases 0}
\end{figure}

\paragraph{Case 1: $\lambda\theta < N\mu$.} This case is represented by the regions A and B that are to the left of the vertical line $\lambda \theta = N\mu$ from Figure \eqref{Fig: queue stability cases 0}. When $\delta \leq \frac{N}{\lambda\theta}$, or region B, the queues approach a stable equilibrium for any delay $\Delta$. Here all eigenvalues stay on the negative (real) side of the complex plane. As $\Delta$ increases, the complex eigenvalues approach the imaginary axis, but never reach it, as shown in Figure \eqref{Fig: eigs1}. However, when $\delta > \frac{N}{\lambda\theta}$,  which is region A of Figure \eqref{Fig: queue stability cases 0}, the queues will never be stable, and will undergo infinitely many Hopf bifurcations as the delay increases. For sufficiently small delay $\Delta$, the complex eigenvalues will be on the positive (real) side of the complex plane, and as $\Delta$ will increase, the complex pairs will cross the imaginary axis from right to left, causing Hopf bifurcations to occur as shown in Figure \eqref{Fig: eigs2}. Note, however, that queues will never gain stability because for any delay $\Delta$ there will be eigenvalues with positive real parts. 
\begin{figure}[H]
  \centering
  \begin{minipage}[b]{0.49\textwidth}
    \includegraphics[width=\textwidth]{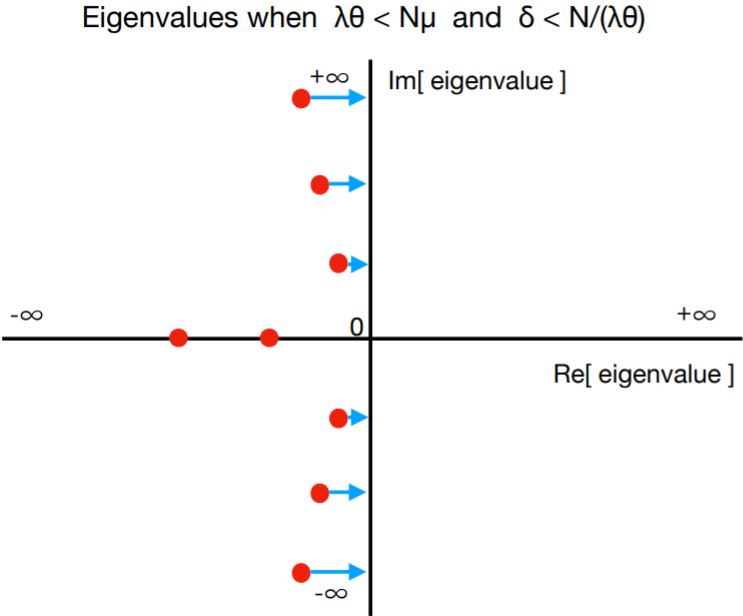}
    \caption{Eigenvalues remain on the left side of the imaginary axis for all $\Delta$. }
    \label{Fig: eigs1}
  \end{minipage}
  \hfill
  \begin{minipage}[b]{0.49\textwidth}
    \includegraphics[width=\textwidth]{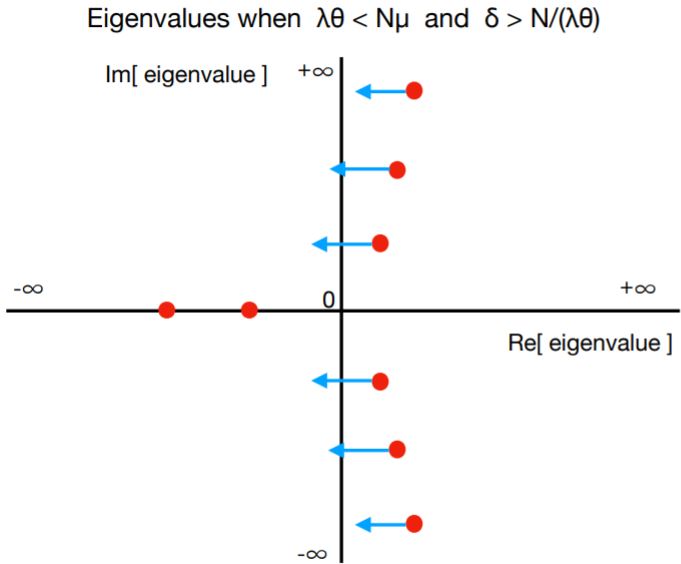}
    \caption{Eigenvalues cross the imaginary axis from right to left as $\Delta$ increases.}
    \label{Fig: eigs2}
  \end{minipage}
\end{figure}
\paragraph{Case 2: $\lambda\theta > N\mu$.} This case is represented by the regions C and D in Figure \eqref{Fig: queue stability cases 0}. When $\delta < \frac{N}{\lambda\theta}$, or region D, the queues will approach a stable equilibrium for a sufficiently small delay $\Delta$. All the eigenvalues will be on the negative (real) side of the complex plane. As the delay $\Delta$ increases, the complex pairs of eigenvalues will move towards the imaginary axis, crossing the axis eventually one by one from left to right as indicated in Figure \eqref{Fig: eigs3}. During the crossing of each pair, a Hopf bifurcation occurs. When $\delta \geq \frac{N}{\lambda\theta}$, which is region C of the diagram \eqref{Fig: queue stability cases 0}, the complex eigenvalues cannot reach the imaginary axis, and they all stay to the right side of the imaginary axis on the complex plane  as show in Figure \eqref{Fig: eigs4}, so there will never be a stable equilibrium. \\
\begin{figure}[H]
  \centering
  \begin{minipage}[b]{0.49\textwidth}
    \includegraphics[width=\textwidth]{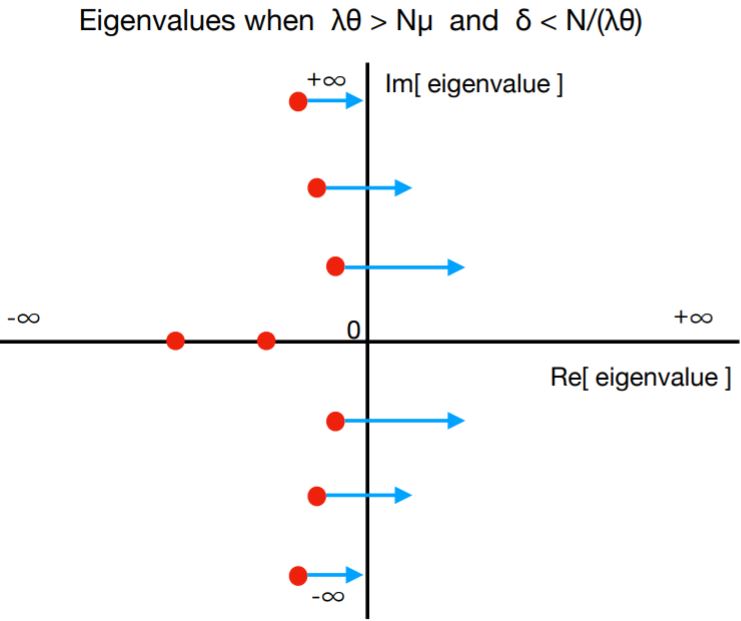}
    \caption{Eigenvalues cross the imaginary axis from left to right as $\Delta$ increases.}
    \label{Fig: eigs3}
  \end{minipage}
  \hfill
  \begin{minipage}[b]{0.49\textwidth}
    \includegraphics[width=\textwidth]{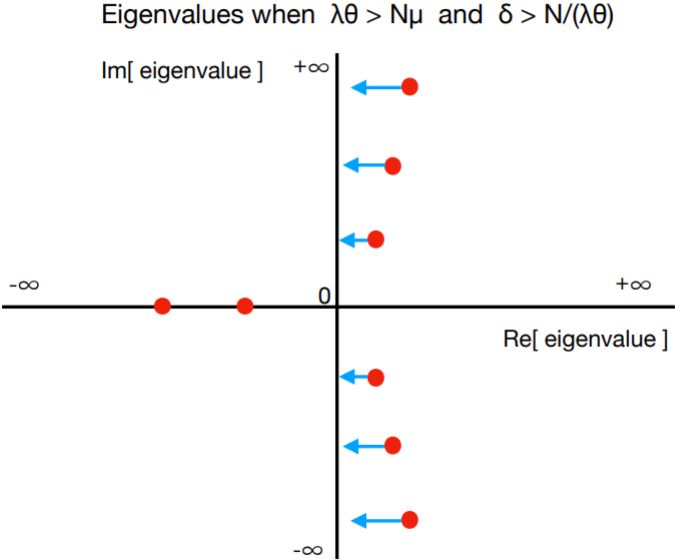}
    \caption{Eigenvalues stay on the right side of imaginary axis for all $\Delta$.}
    \label{Fig: eigs4}
  \end{minipage}
\end{figure}  

Another aspect to point out is the dependence on the MNL parameter $\theta$. When customers join the queues at random, or $\theta \to 0$, the parameters inevitably end up in region B of Figure  \eqref{Fig: queue stability cases 0}, so the queues will stable for any delay. Alternatively, if customers always join the shortest queue, or $\theta \to \infty$, then for any $\delta >0$ we inevitably end up in region C of Figure  \eqref{Fig: queue stability cases 0}, so the queues will always be unstable.




\section{Achieving Maximum Stability} \label{Sec: comparison}
In physical settings, it is often important to preserve the stability of the queues. Stability evens out the individual waiting times of the customers, minimizing the negative experience. It is therefore useful to know when providing extra information helps to postpone the point of the bifurcation, and when the extra information makes the bifurcation to happen sooner. For example, consider a numerical example from Figures \eqref{Fig: before Hopf 2} - \eqref{Fig: after Hopf}, with two queues and fixed parameters $\lambda, \theta, \mu,$ and $\Delta$. In Figures \eqref{Fig: before Hopf 2} and \eqref{Fig: before Hopf}, $\Delta <\Delta_{cr}$ so the queues converge to an equilibrium over time. However, in Figures \eqref{Fig: after Hopf 2} and \eqref{Fig: after Hopf} we have $\Delta >\Delta_{cr}$, so the queues oscillate indefinitely. Although the delay $\Delta$ is the same, the change of behavior results from tweaking the parameter $\delta$, which consequently regulates the bifurcation threshold $\Delta_{cr}$.

\begin{figure}[H]
  \centering
  \begin{minipage}[b]{0.49\textwidth}
    \includegraphics[width=\textwidth]{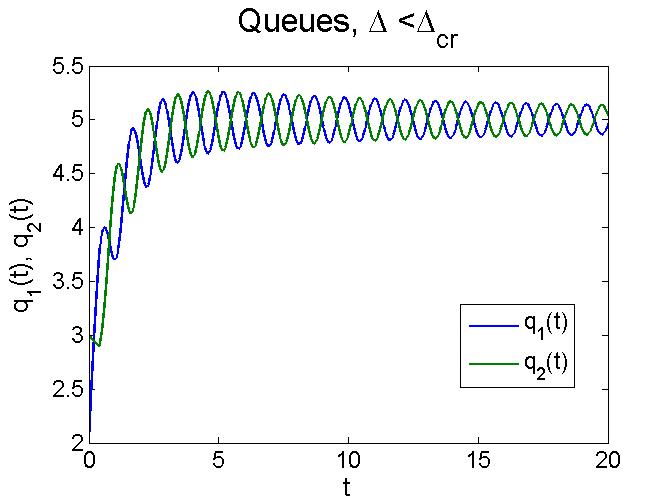}
    \caption{Queues before Hopf bifurcation;\\
    $\delta = 0.08,$ $\lambda= 10$, $\mu = 1$, $\theta = 1.$}
    \label{Fig: before Hopf 2}
  \end{minipage}
  \hfill
  \begin{minipage}[b]{0.49\textwidth}
    \includegraphics[width=\textwidth]{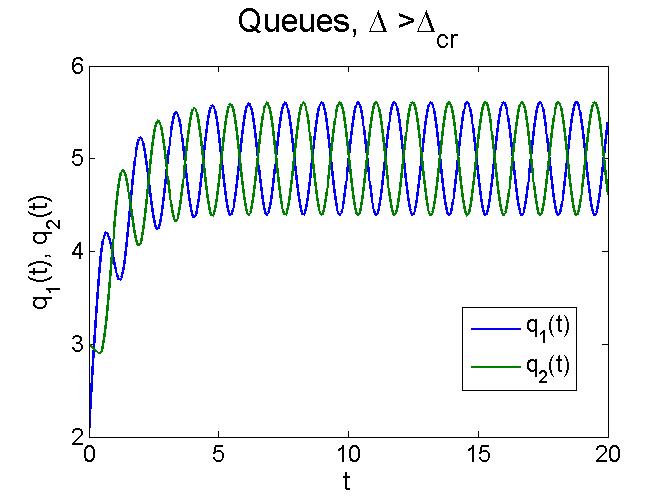}
    \caption{Queues after Hopf bifurcation; \\
    $\delta = 0,$ $\lambda= 10$, $\mu = 1$, $\theta = 1.$}
    \label{Fig: after Hopf 2}
  \end{minipage}
\end{figure}

\begin{figure}[H]
  \centering
  \begin{minipage}[b]{0.49\textwidth}
    \includegraphics[width=\textwidth]{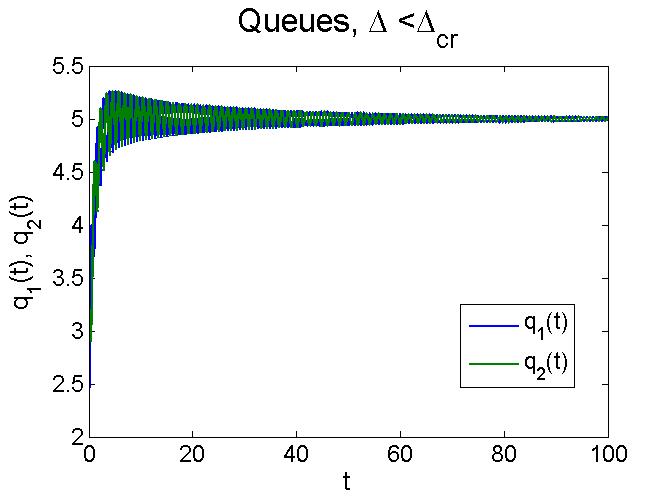}
    \caption{Queues before Hopf bifurcation; \\ $\delta = 0.08,$ $\lambda= 10$, $\mu = 1$, $\theta = 1.$ }
    \label{Fig: before Hopf}
  \end{minipage}
  \hfill
  \begin{minipage}[b]{0.49\textwidth}
    \includegraphics[width=\textwidth]{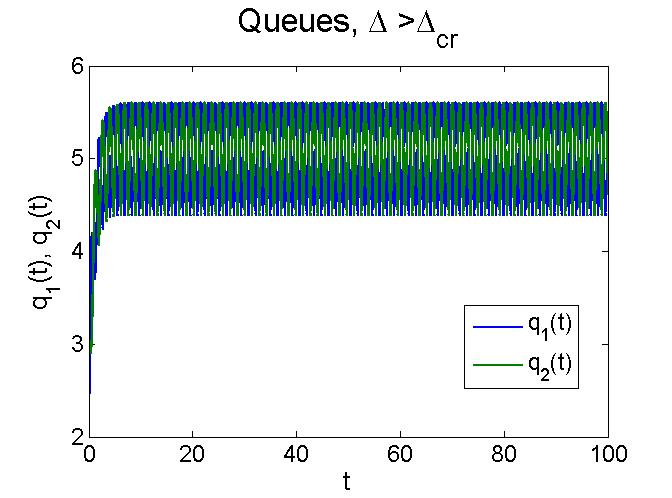}
    \caption{Queues after Hopf bifurcation;\\ $\delta = 0,$ $\lambda= 10$, $\mu = 1$, $\theta = 1.$}
    \label{Fig: after Hopf}
  \end{minipage}
\end{figure}

In this Section, we will consider the scenario $\lambda\theta > N \mu$, where the equilibrium of the queues can become unstable.  We will study how the bifurcation threshold $\Delta_{cr}$ changes depending on the weight of the velocity information $\delta$. Our next result shows that the threshold $\Delta_{cr}$ is a concave function of $\delta$.

\begin{proposition} \label{Delta cr concavity}
Suppose $\lambda\theta >N\mu$. Then the function $\Delta_{cr}(\delta)$ is concave for all $\delta \in [0 , \frac{N}{\lambda \theta})$.
\end{proposition}
\begin{proof}
The critical delay $\Delta_{cr}$ is given by Equation \eqref{delta cr}. It is clear that the second derivative $\frac{d^2\Delta_{cr}}{d\delta^2}$ is negative for all $\delta \in [0 , \frac{N}{\lambda \theta})$:
\begin{eqnarray}
\frac{d^2\Delta_{cr}}{d\delta^2} = -\frac{1 }{C_3} \cdot \bigg(C_1+C_2 \arccos \bigg(- \frac{\delta \lambda^2\theta^2 + N^2 \mu}{N \lambda\theta(1 + N \delta  \mu)}\bigg) \bigg), \quad \text{where} \\
C_1 = (N^2 - \delta^2\lambda^2\theta^2)(\lambda^2\theta^2 - N^2 \mu^2)(\delta \lambda^2\theta^2 + N^2\mu) >0\\
C_2 = N^2\lambda^2\theta^2(1+\delta\mu)^2\sqrt{(N^2 - \delta^2\lambda^2\theta^2)(\lambda^2\theta^2 - N^2 \mu^2)} >0 \\
C_3 = N\lambda\theta (N^2 - \delta^2\lambda^2\theta^2)^{\frac{3}{2}} \sqrt{\lambda^2\theta^2 - N^2\mu^2}(1+\delta\mu)^3 \sqrt{1 - \frac{(\delta\lambda^2\theta^2 + N^2\mu)^2}{(N\lambda\theta + N\delta\lambda\theta\mu)^2}}>0\\
\arccos \bigg(- \frac{\delta \lambda^2\theta^2 + N^2 \mu}{N \lambda\theta(1 +  \delta \mu)}\bigg) = \Delta_{cr}\cdot \sqrt{\frac{\lambda^2\theta^2 - N^2\mu^2}{N^2 - \delta^2\lambda^2\theta^2}} > 0. 
\end{eqnarray}
\end{proof}

Proposition \eqref{Delta cr concavity} allows us to show that there exists a specific size of the weight $\delta$ that makes the queueing system optimally stable. We call this size of the weight $\delta_{max}$, and it is such that $\delta=\delta_{max}$ maximizes the threshold $\Delta_{cr}$. In Proposition \eqref{proposition: delta max}, we give an equation that determines $\delta_{max}$, and provide closed-form expressions for an upper and a lower bound of $\delta_{max}$.
\begin{proposition} \label{proposition: delta max}
Suppose $\lambda \theta > N \mu$. There exists a unique $\delta_{max} \geq 0$ that maximizes a given root $\Delta_{cr}$ for fixed parameters $\lambda,\mu, N, \theta$. It is given by solution of
\begin{eqnarray}
\frac{\sqrt{N^2 - \delta_{max}^2\lambda^2\theta^2}}{1+\delta_{max}\mu} = \frac{\delta_{max}\lambda^2\theta^2}{\sqrt{\lambda^2\theta^2 - N^2\mu^2}}\cdot \arccos \bigg(- \frac{\delta_{max} \lambda^2\theta^2 + N^2 \mu}{N \lambda\theta(1 + \delta_{max} \mu)}\bigg). \label{delta max eqn}
\end{eqnarray}
Furthermore, $\delta_{max}$ is bounded by $\delta_1 < \delta_{max} < \delta_2$, where
\begin{eqnarray}
\delta_1 &=& \frac{-\big(\Delta_0 + \frac{N}{\lambda\theta}\big) \lambda\theta + \sqrt{\lambda^2\theta^2\big(\Delta_0 + \frac{N}{\lambda\theta}\big)^2 + 4N^2\big(\Delta_0 + \frac{N}{\lambda\theta}\big)\mu + 4N^2}}{2\lambda\theta\Big(1 + \big(\Delta_0 + \frac{N}{\lambda\theta}\big)\mu\Big)} \label{delta_1}\\
\delta_{2} &=& \frac{-\Delta_0\lambda\theta + \sqrt{\lambda^2\theta^2\Delta_0^2 + 4N^2\Delta_0\mu + 4N^2}}{2\lambda\theta(1 + \Delta_0\mu)} \label{delta_2}\\
\Delta_0 &=& \arccos \bigg(- \frac{N \mu}{ \lambda\theta}\bigg) \cdot \sqrt{\frac{N^2 }{\lambda^2\theta^2 - N^2\mu^2}}. 
\end{eqnarray}
\end{proposition}
\begin{proof}
We can treat $\Delta_{cr}$ as a function of $\delta$. The implicit differentiation of \eqref{Hopf condition} gives the rate with which $\Delta_{cr}$ changes: 
\begin{eqnarray}
\frac{d}{d\delta}\Delta_{cr}(\delta) = \frac{N^2 - \delta \lambda^2\theta^2\Big(\delta + \Delta_{cr}(\delta) + \delta\mu\Delta_{cr}(\delta)\Big)}{(N^2 - \delta^2\lambda^2\theta^2)(1+ \delta \mu)} = \frac{1}{1+\delta\mu}-  \frac{\delta\lambda^2\theta^2\Delta_{cr}(\delta)}{N^2 - \delta^2\lambda^2\theta^2}  \label{Delta cr partial 1}\\
=  \frac{1}{1+\delta\mu} -\frac{\delta\lambda^2\theta^2}{N^2 - \delta^2\lambda^2\theta^2}\cdot \arccos \bigg(- \frac{\delta \lambda^2\theta^2 + N^2 \mu}{N \lambda\theta(1 + \delta \mu)}\bigg) \cdot \sqrt{\frac{N^2 - \delta^2\lambda^2\theta^2}{\lambda^2\theta^2 - N^2\mu^2}}. \label{Delta cr partial 2}
\end{eqnarray}

By Proposition \eqref{Delta cr concavity}, $\Delta_{cr}(\delta)$ is concave on the interval $[0,\frac{N}{\lambda\theta})$. Further, it can be shown that $\frac{d}{d\delta}\Delta_{cr}(0) = 1 >0$ and $\lim_{\delta \to \frac{N}{\lambda\theta}}\frac{d}{d\delta}\Delta_{cr}(\delta) = - \infty <0$, so there is a point $\delta_{max}$ where $\frac{d}{d\delta}\Delta_{cr}(\delta_{max}) = 0$. Therefore $\Delta_{cr}(\delta)$ reaches its absolute maximum at $\delta_{max} \in (0,\frac{N}{\lambda\theta})$. For intuition, we plot $\frac{d}{d\delta}\Delta_{cr}(\delta)$ in Figure \eqref{Fig: Delta_cr derivative}. 

\begin{figure}[H]
  \centering
    \includegraphics[scale=.65]{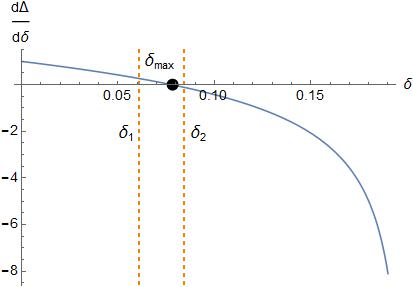}
    \caption{$\delta_{max}$ and its bounds $\delta_1<\delta_{max}< \delta_2$. }
    \label{Fig: Delta_cr derivative}
 \end{figure}
The value $\delta_{max}$ can be found numerically by solving $\frac{d}{d\delta}\Delta_{cr}(\delta_{max}) = 0$ from Equation \eqref{Delta cr partial 2}, alternatively written as \eqref{delta max eqn}. It is left to find closed-form expressions for the bounds on $\delta_{max}$. By Equation  \eqref{Delta cr partial 1}, we can express $\delta_{max}$ as
\begin{eqnarray}
\frac{d}{d\delta}\Delta_{cr}(\delta_{max}) =\frac{1}{1+\delta_{max}\mu}-  \frac{\delta_{max}\lambda^2\theta^2\Delta_{cr}(\delta_{max})}{N^2 - \delta_{max}^2\lambda^2\theta^2} = 0, \label{bound 1} \\
\frac{1}{1+\delta_{max}\mu}-  \frac{\delta_{max}\lambda^2\theta^2\Delta_0}{N^2 - \delta_{max}^2\lambda^2\theta^2}>0, \label{ub 1}
\end{eqnarray}
where $\Delta_0 = \Delta_{cr}(0) <\Delta_{cr}(\delta_{max})$. When solved for $\delta_{max}$, the inequality \eqref{ub 1} produces an upper bound condition $\delta_{max} < \delta_2$ given by Equation \eqref{delta_2}.

To find the lower bound, we note that $\frac{d}{d\delta}\Delta_{cr}(\delta)$ is monotonically decreasing. Thus, $\frac{d}{d\delta}\Delta_{cr}(\delta) < \frac{d}{d\delta}\Delta_{cr}(0) =1$ for all $\delta \in (0, \frac{N}{\lambda\theta})$, and $\Delta_{cr}(\delta) \leq \delta +\Delta_{cr}(0) < \frac{N}{\lambda\theta} +\Delta_0$. Therefore by Equation \eqref{bound 1}, we get
\begin{eqnarray}
\frac{1}{1+\delta_{max}\mu}-  \frac{\delta_{max}\lambda^2\theta^2\big(\Delta_0 + \frac{N}{\lambda\theta}\big) }{N^2 - \delta_{max}^2\lambda^2\theta^2}<0,
\end{eqnarray}
which produces the bound $\delta_{max} > \delta_1$ from Equation \eqref{delta_1} when solved for $\delta_{max}$.
\end{proof}
Figures \eqref{Fig: Delta cr lambda delta} - \eqref{Fig: Delta cr lambda delta big} show $\Delta_{cr}$ as a function of $\lambda$ and $\delta$. For each arrival rate $\lambda$, the maximum $\Delta_{cr}$ is attained for some $\delta$ between the two curves $\delta_1$ and $\delta_2$. Similarly, Figures \eqref{Fig: Delta cr mu delta} and \eqref{Fig: Delta cr mu delta big} show $\Delta_{cr}$ as a function of $\mu$ and $\delta$ with the two curves $\delta_1$ and $\delta_2$. As seen in the Figures \eqref{Fig: Delta cr lambda delta} - \eqref{Fig: Delta cr mu delta big}, the bounds on $\delta_{max}$ are tight.

\begin{figure}[H]
  \centering
  \begin{minipage}[b]{0.49\textwidth}
    \includegraphics[width=\textwidth]{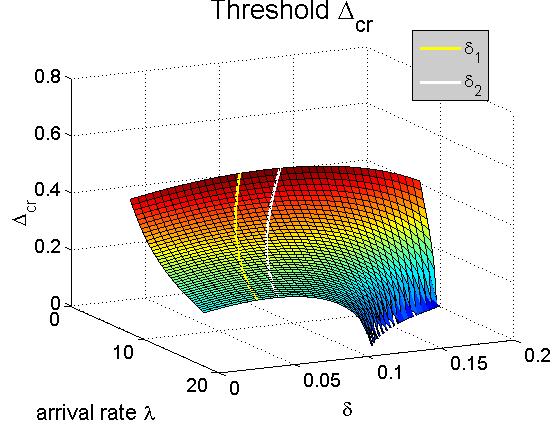}
    \caption{For each $\lambda$, the maximum $\Delta_{cr}$ is achieved when $\delta \in (\delta_1,\delta_2)$; $\mu =1,$ $\theta = 1.$}
    \label{Fig: Delta cr lambda delta}
  \end{minipage}
  \hfill
  \begin{minipage}[b]{0.49\textwidth}
    \includegraphics[width=\textwidth]{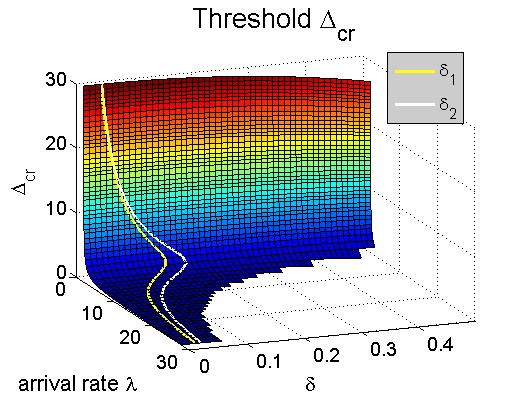}
    \caption{For each $\lambda$, the maximum $\Delta_{cr}$ is achieved when $\delta \in (\delta_1,\delta_2)$; $\mu =1,$ $\theta = 1.$}
    \label{Fig: Delta cr lambda delta big}
  \end{minipage}
\end{figure}

\begin{figure}[H]
  \centering
  \begin{minipage}[b]{0.49\textwidth}
    \includegraphics[width=\textwidth]{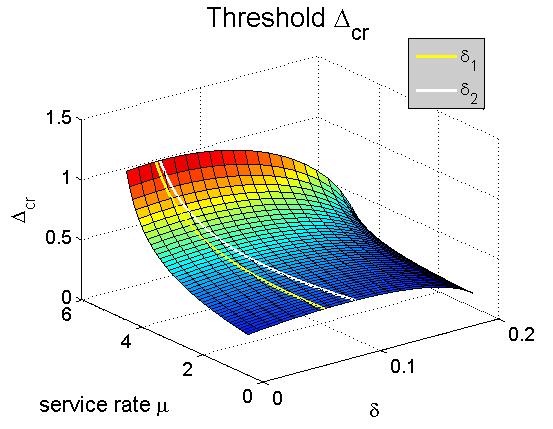}
    \caption{For each $\mu$, the maximum $\Delta_{cr}$ is achieved when $\delta \in (\delta_1,\delta_2)$; $\lambda =10,$ $\theta = 1.$}
    \label{Fig: Delta cr mu delta}
  \end{minipage}
  \hfill
  \begin{minipage}[b]{0.49\textwidth}
    \includegraphics[width=\textwidth]{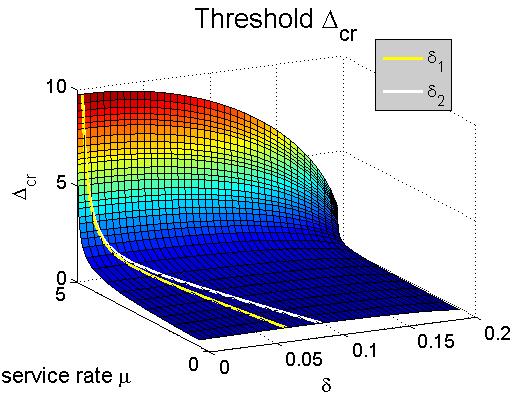}
    \caption{For each $\mu$, the maximum $\Delta_{cr}$ is achieved when $\delta \in (\delta_1,\delta_2)$; $\lambda =10,$ $\theta = 1.$}
    \label{Fig: Delta cr mu delta big}
  \end{minipage}
\end{figure}

Besides knowing at which value $\delta$ the maximal bifurcation threshold $\Delta_{cr}$ may occur, it is also important to know how large that threshold actually is. In the next result, we develop bounds for the maximum $\Delta_{cr}$ that can be attained for fixed parameters $\lambda, N, \mu,$ and $\theta$.

\begin{proposition} \label{proposition: Delta_cr bounds}
The maximum value of a root $\Delta_{cr}$ for fixed parameters $\lambda, \mu$, and $N$, is attained at $\delta_{max}$ and is bounded by $\Delta_{1}< \Delta_{cr}(\delta_{max}) < \Delta_2$, where 
\begin{eqnarray}
\Delta_1 = \max[\Delta_{cr}(\delta_1), \Delta_{cr}(\delta_2)] , \quad \Delta_2 = \min[\Delta_{2a}, \Delta_{2b}], \\
\Delta_{2a} = \Delta_{cr}(\delta_1) + (\delta_2 - \delta_1)\cdot\frac{d}{d\delta}\Delta_{cr}(\delta_1), \quad  \Delta_{2b} = \Delta_{cr}(\delta_2) - (\delta_2 - \delta_1)\cdot\frac{d}{d\delta}\Delta_{cr}(\delta_2).
\end{eqnarray}
\end{proposition}
\begin{proof}

By Proposition \eqref{proposition: delta max}, $\Delta_{cr}(\delta)$ attains its maximum at $\delta = \delta_{max}$. Hence the lower bound $\Delta_{1}  < \Delta_{cr}(\delta_{max})$ trivially follows, since $\delta_{max} \neq \delta_1,\delta_2$. To find an upper bound, note that $\frac{d}{d\delta}\Delta_{cr}(\delta)$ is a monotonically decreasing function, so $\frac{d}{d\delta}\Delta_{cr}(\delta_1) > \frac{d}{d\delta}\Delta_{cr}(\delta)$ for all $\delta>\delta_1$, and also that $\frac{d}{d\delta}\Delta_{cr}(\delta_1) >0$ since $\Delta_{cr}(\delta)$ increases while $\delta < \delta_{max}$. Hence 
\begin{eqnarray}
\Delta_{cr}(\delta_{max}) = \Delta_{cr}(\delta_{1}) + \int_{\delta_1}^{\delta_{max}} \frac{d}{d\delta}\Delta_{cr}(\delta) d\delta < \Delta_{cr}(\delta_{1}) + \int_{\delta_1}^{\delta_{max}} \frac{d}{d\delta}\Delta_{cr}(\delta_1) d\delta  \\
= \Delta_{cr}(\delta_{1}) + (\delta_{max} - \delta_1) \frac{d}{d\delta}\Delta_{cr}(\delta_1) <\Delta_{cr}(\delta_{1}) + (\delta_2 - \delta_1) \frac{d}{d\delta}\Delta_{cr}(\delta_1) = \Delta_{2a}.
\end{eqnarray}
In addition, it is known that $\frac{d}{d\delta}\Delta_{cr}(\delta)<0$ when $\delta>\delta_{max}$, so 
\begin{eqnarray}
\Delta_{cr}(\delta_{max}) = \Delta_{cr}(\delta_{2}) - \int_{\delta_{max}}^{\delta_{2}} \frac{d}{d\delta}\Delta_{cr}(\delta) d\delta < \Delta_{cr}(\delta_{2}) - \int_{\delta_{max}}^{\delta_{2}} \frac{d}{d\delta}\Delta_{cr}(\delta_2) d\delta \\
= \Delta_{cr}(\delta_{2}) - (\delta_2 - \delta_{max}) \frac{d}{d\delta}\Delta_{cr}(\delta_2) 
< \Delta_{cr}(\delta_{2}) - (\delta_2 - \delta_{1}) \frac{d}{d\delta}\Delta_{cr}(\delta_2)  = \Delta_{2b}.
\end{eqnarray}
Therefore $\Delta_{cr}(\delta_{max}) < \min[\Delta_{2a},\Delta_{2b}] = \Delta_2$, as desired.
\end{proof}

Figure \eqref{Fig: Delta_cr} illustrates $\Delta_{cr}(\delta) - \Delta_{cr}(0)$ as a function of $\delta$, with the maximum attained at $\delta_{max}$ and the bounds on the maximum given by $\Delta_1$ and $\Delta_2$. Further, it is evident from Figure \eqref{Fig: Delta_cr} that there is a threshold value, which we call $\delta_{cap}$, that places a cap on the potential utility of the velocity information. When $\delta$ is less than $\delta_{cap}$, the queueing system becomes more stable from the velocity information because $\Delta_{cr}(\delta) > \Delta_{cr}(0)$. However, when $\delta$ exceeds $\delta_{cap}$, the queues become more unstable in the sense that $\Delta_{cr}(\delta) < \Delta_{cr}(0)$. The result below provides an equation for $\delta_{cap}$.

 \begin{figure}[H]
   \centering
    \includegraphics[scale=.75]{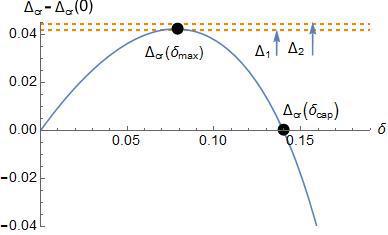}
    \caption{$\delta_{max}$ maximizes $\Delta_{cr}$.}
    \label{Fig: Delta_cr}
\end{figure}

\begin{proposition} \label{proposition: delta cap} Suppose $\lambda \theta > N \mu$. There exists a unique $\delta_{cap} > 0$ such that $\Delta_{cr}(\delta) > \Delta_{cr}(0)$ for all $\delta < \delta_{cap}$, and $\Delta_{cr}(\delta) < \Delta_{cr}(0)$ for all $\delta > \delta_{cap}$. It is given by the solution to
\begin{eqnarray}
\arccos \bigg(- \frac{ N \mu}{\lambda\theta}\bigg) \sqrt{\frac{N^2 }{\lambda^2\theta^2 - N^2\mu^2}} = \arccos \bigg(- \frac{\delta_{cap} \lambda^2\theta^2 + N^2 \mu}{N \lambda\theta(1 + \delta_{cap} \mu)}\bigg) \sqrt{\frac{N^2 - \delta_{cap}^2\lambda^2\theta^2}{\lambda^2\theta^2 - N^2\mu^2}}. \label{delta_cap}
\end{eqnarray}
\end{proposition}
\begin{proof}
As previously shown, $\Delta_{cr}(\delta)$ is monotonically increasing on $\delta \in [0, \delta_{max})$ and monotonically decreasing on $\delta \in (\delta_{max}, \frac{N}{\lambda \theta})$. Further, $\lim_{\delta \to \frac{N}{\lambda \theta}}\Delta_{cr}(\delta) = 0<\Delta_{cr}(0)$ since $\Delta_{cr}(0)>0$ by assumption, so there exists exactly one point $\delta_{cap}$ on the interval $(\delta_{max}, \frac{N}{\lambda \theta})$ where $\Delta_{cr}(\delta_{cap}) = \Delta_{cr}(0)$, and it also follows that $\Delta_{cr}(\delta_{cap}) > \Delta_{cr}(0)$ for all $\delta<\delta_{cap}$ and $\Delta_{cr}(\delta_{cap}) < \Delta_{cr}(0)$ for all $\delta>\delta_{cap}$. By substituting the expression for $\Delta_{cr}$ from \eqref{delta cr} into  $\Delta_{cr}(0) - \Delta_{cr}(\delta_{cap}) = 0 $ we get Equation \eqref{delta_cap}.
\end{proof}

To summarize, when $\lambda\theta >N\mu$, the queues are stable when the delay is less than $\Delta_{cr}$. We can therefore provide the most stability for the queues by choosing $\delta$ that maximizes $\Delta_{cr}$, i.e. $\delta_{max}$.  Proposition \eqref{proposition: delta max} proves the existence of $\delta_{max}$, gives an equation describing $\delta_{max}$, and provides closed-form expressions for bounds $\delta_1$ and $\delta_2$ such that $\delta_1<\delta_{max}<\delta_2$. Proposition \eqref{proposition: Delta_cr bounds} also provides bounds $\Delta_1$ and $\Delta_2$ for the maximum value that $\Delta_{cr}$ can take as a function of $\delta$, so $\Delta_1 < \Delta_{cr}(\delta_{max}) < \Delta_2$. Lastly, we show that even if $\delta \neq \delta_{max}$, it is still beneficial to include the velocity information as long as $\delta < \delta_{cap}$. When $\delta$ exceeds $\delta_{cap} $, however, $\Delta_{cr}(\delta)$ becomes less than $\Delta_{cr}(0)$, so the queues are less likely to be stable than if the velocity information was omitted altogether. Proposition \eqref{proposition: delta cap} proves the existence of $\delta_{cap}$ and provides the equation for it. 


\section{Impact of Velocity Information on the Amplitude} \label{Sec: N=2}

Now that we have a good understanding of how the velocity information impacts the critical delay, we address a more practical question.  What is the impact of the velocity on the amplitude of the oscillations?  This question is important because it reveals how much the queues will oscillate when they are not in equilibrium.  Moreover, it can provide an estimate of how much throughput is lost because of the oscillations (amusement park capacity) or even provide valuable estimates of how much fuel or energy is lost in transportation settings. 

Although our previous analysis holds for an arbitrary number of queues, in the sequel, we will demonstrate how $\delta$ affects the amplitude dynamics of the queues in the case of a two queue network.  One reason for this restriction is that we must move beyond linearization techniques.  In fact, we must use third order Taylor expansions to obtain information about the amplitude, see for example \cite{novitzky2018}.  Thus, many of the matrix techniques we exploited for linearizing the NDDE in Section (2), cannot be used in the context of tensors for the third order Taylor expansion.   Thus, for the case of two dimensions, we have the following system of equations
\begin{eqnarray}
\label{q_1 dot}
\shortdot{q}_1(t) =  \lambda \cdot \frac{\exp\Big(-\theta\big(q_1(t - \Delta)+ \delta \shortdot{q}_1(t - \Delta)\big)\Big)}{\sum_{j = 1}^2 \exp\Big(-\theta\big(q_j(t - \Delta) + \delta \shortdot{q}_j(t - \Delta)\big)\Big)}  - \mu q_1(t)\\
\label{q_2 dot}
\shortdot{q}_2(t) =  \lambda \cdot \frac{\exp\Big(-\theta\big(q_2(t - \Delta)+ \delta \shortdot{q}_2(t - \Delta)\big)\Big)}{\sum_{j = 1}^2 \exp\Big(-\theta\big(q_j(t - \Delta) + \delta \shortdot{q}_j(t - \Delta)\big)\Big)} - \mu q_2(t),
\end{eqnarray}
where as usual $\Delta, \lambda,\mu , \theta> 0$ and $\delta \geq 0$. Similar to Section \eqref{Sec: comparison}, we will consider the scenario with $\lambda\theta > 2\mu$, where even for a small $\delta$ the equilibrium of queues loses stability for sufficiently large delay. Our first result shows that the Hopf bifurcations that occur at each root $\Delta_{cr}$ are supercritical.  

\begin{theorem}\label{Theorem: Hopf stability}
Suppose $\omega_{cr} \in \mathbb{R}$ and $\omega_{cr} \neq 0$. The NDDE system \eqref{q_1 dot} - \eqref{q_2 dot} undergoes a supercritical Hopf bifurcation at each root $\Delta_{cr}$. If $\delta\mu <1$ then the limit cycle is born when $\Delta \leq \Delta_{cr}$. If $\delta\mu >1$ then the limit cycle is born when $\Delta \geq \Delta_{cr}$.
\end{theorem}

\begin{proof}
We will use the method of {\it slow flow}, or the Method of Multiple Scales, to determine the stability of the Hopf bifurcations given by Theorem \eqref{thm: Hopf}. This method is often applied to systems of delay differential equations (DDEs) \cite{das2002,belhaq2008,lazarus2017}. We note, however, that the stability of the limit cycles can also be determined by showing that the floquet exponent has negative real part, as outlined in Hassard et al. \cite{hassard1981theory}. 

The first step in the method of slow flow is to consider the perturbation of $q_1$ and $q_2$ from the equilibrium point $q_1^* = q_2^* = \frac{\lambda}{2 \mu}$, and to approximate the resulting derivatives by third order Taylor expansion. The two resulting DDEs can be uncoupled when their sum and their difference are taken
\begin{eqnarray}
w_1(t) = q_1(t) + q_2(t), \quad w_2(t) = q_1(t) - q_2(t)  \label{w1 w2}\\
\shortdot{w}_1(t) = -\mu w_1(t) \\ \label{w2}
\shortdot{w}_2(t) = -\mu w_2(t) -\frac{\lambda\theta}{2}(w_2(t - \Delta) + \delta \shortdot{w}_2(t-\Delta))\\
+ \frac{\lambda\theta^3}{24}(w_2(t - \Delta) + \delta \shortdot{w}_2(t-\Delta))^3 + O(w_2^4).
\end{eqnarray}
The function $w_1(t) = C e^{-\mu t}$ decays to $0$, while the function $w_2(t)$  has a Hopf bifurcation at $\Delta_{cr}$ where the periodic solutions are born. 

We set $w_2(t) = \sqrt{\epsilon } x(t)$ in order to prepare the NDDE for perturbation treatment:
\begin{equation}\label{x}
\shortdot{x}(t) = -\mu x(t) -\frac{\lambda\theta}{2}(x(t - \Delta) + \delta \shortdot{x}(t-\Delta)) + \frac{\sqrt{\epsilon}\lambda\theta^3}{24}(x(t - \Delta) + \delta \shortdot{x}(t-\Delta))^3
\end{equation}
We replace the independent variable t by two new time variables
$\xi = \omega t$ (stretched time) and $\eta = \epsilon t$ (slow time).
Then we expand $\Delta$ and $\omega$ about the critical Hopf values:
\begin{align}
\Delta = \Delta_{cr} + \epsilon \alpha, \quad \omega = \omega_{cr} + \epsilon \beta.
\end{align}
The time derivative $\shortdot{x}$ becomes
\begin{equation}
\shortdot{x} = \frac{dx}{dt} = \frac{\partial x}{\partial \xi}\frac{d \xi}{dt} + \frac{\partial x}{\partial \eta}\frac{d \eta}{dt} = \frac{\partial x}{\partial \xi}\cdot (\omega_{cr} + \epsilon \beta) + \frac{\partial x}{\partial \eta}\cdot \epsilon.
\end{equation}
The expression for $x(t-\Delta)$ may be simplified by Taylor expansion for small $\epsilon$:
\begin{align}
x(t - \Delta) &= x(\xi - \omega \Delta, \eta - \epsilon \Delta) \\
&= x(\xi - (\omega_{cr} + \epsilon \beta) (\Delta_{cr} + \epsilon \alpha), \eta - \epsilon (\Delta_{cr} + \epsilon \alpha)) + O(\epsilon^2) \\
&=  \tilde{x} - \epsilon(\omega_{cr}\alpha + \Delta_{cr}\beta)\cdot \frac{\partial  \tilde{x}}{\partial \xi}  - \epsilon \Delta_{cr}\frac{\partial  \tilde{x}}{\partial \eta} + O(\epsilon^2), \label{x delayed}
\end{align}
where $ x(\xi - \omega_{cr}\Delta_{cr}, \eta) = \tilde{x}$.  The function $x$ is represented as $x = x_0 + \epsilon x_1 + \dots$, and we get
\begin{equation} \label{dx dt}
\frac{dx}{dt} = \omega_{cr} \frac{\partial x_0}{\partial \xi}  + \epsilon \beta \frac{\partial x_0}{\partial \xi} + \epsilon \frac{\partial x_0}{\partial \eta} + \epsilon \omega_{cr} \frac{\partial x_1}{\partial \xi} .
\end{equation}
After these substitutions are made into \eqref{x}, the resulting equation can be separated by the powers of $\epsilon$ into two equations. For the $\epsilon^0$ terms, we get an equation for $x_0$ without any terms involving $x_1$, namely $L(x_0)=0$, where
\begin{eqnarray}
L(x_0) = \mu x_0 +  \frac{\lambda}{2} \tilde{x}_0 + \omega_{cr} \frac{\partial x_0}{\partial \xi} +  \frac{\delta \lambda \omega_{cr}}{2}  \frac{\partial \tilde{x}_0}{\partial \xi} = 0,
\end{eqnarray}
which is satisfied with a solution of the form 
\begin{eqnarray} \label{x_0}
x_0(t) = A(\eta)\cos(\xi) + B(\eta) \sin(\xi).
\end{eqnarray}
The equation resulting from $\epsilon^1$ terms is $L(x_1) + M(x_0) = 0$. Since $L(x_1) = 0$ is satisfied by a solution of the form \eqref{x_0}, then the terms from $M(x_0)$ involving  $\cos(\xi)$ and $\sin(\xi)$ are resonant. To eliminate the resonant terms, their coefficients must be $0$, which gives two equations for $A(\eta)$ and $B(\eta)$. Switching into polar coordinates, we define $R = \sqrt{A^2 + B^2}$, and find 
\begin{eqnarray} \label{dR deta}
\frac{d R}{d \eta}& = &- \frac{R\big( c_1 R^2 - c_2 \big)}{c_3},\quad  \text{ where} \\
c_1 &=&(\mu^2 + \omega_{cr}^2) (\mu^2 + \omega_{cr}^2 + \delta^2 \omega_{cr}^2\mu^2+ \delta^2\omega_{cr}^4 )\Delta_{cr} \\
&+& (\mu^2 + \omega_{cr}^2)(\mu - \delta \mu^2 + \delta^2 \omega_{cr}^2 \mu - \delta\omega_{cr}^2),  \\
c_2 &=& 4 \alpha \lambda^2\omega_{cr}^2 (1 - \delta^2\mu^2), \\
c_3 &=& \Delta_{cr}^2\cdot 4\lambda^2(\mu^2 + \omega_{cr}^2+ \delta^2\mu^2\omega^2 + \delta^2\omega_{cr}^2) \\
&+& \Delta_{cr} \cdot  8\lambda^2(\mu - \delta\mu^2 - \delta \omega_{cr}^2 + \delta^2\mu \omega_{cr}^2) +  4\lambda^2(1 - \delta\mu)^2. 
\end{eqnarray}
In order to find the equilibrium points of $R$ and to discuss their stability, we need to show that the coefficients $c_1,c_2$, and $c_3$ are positive. Notice that $c_3$ is a quadratic function of $\Delta_{cr}$ with the minimum located at $\Delta_{cr}^*$ such that $\frac{d }{d \Delta_{cr}} c_3(\Delta_{cr}^*) = 0$, hence
\begin{eqnarray}
\Delta_{cr}^* = \frac{\delta}{1 + \delta^2\omega_{cr}^2} - \frac{\mu}{\mu^2 + \omega_{cr}^2} \\
c_3 = c_3(\Delta_{cr}) \geq c_3(\Delta_{cr}^*) = \frac{4\lambda^2 \omega_{cr}^2(1 - \delta^2\mu^2)^2}{(\mu^2 + \omega_{cr}^2)(1+\delta^2\omega_{cr}^2)} >0
\end{eqnarray}
therefore the denominator of $\frac{d R}{d \eta}$, $c_3$, is always positive. Also, we can show $c_1$ to be positive. We first note that at the Hopf, Equation \eqref{cos Hopf} must be satisfied so $\cos(\omega_{cr}\Delta_{cr}) <0$, which implies that  $\omega_{cr}\Delta_{cr} > \frac{\pi}{2}$ and 
\begin{eqnarray} \label{delta_cr pi/2 bound}
\Delta_{cr} > \frac{\pi}{2\omega_{cr}}.
\end{eqnarray} 
Next, we note that $c_1$ is an increasing linear function of $\Delta_{cr}$, so $c_1$ must be positive for any $\Delta_{cr} >\Delta_{cr}^*$ where $c_1(\Delta_{cr}^*) = 0$. This $\Delta_{cr}^*$ is found to be
\begin{eqnarray}
\Delta_{cr}^* = \frac{\delta}{1 + \delta^2\omega_{cr}^2} - \frac{\mu}{\mu^2 + \omega_{cr}^2}.
\end{eqnarray}
Using the inequality in \eqref{delta_cr pi/2 bound}, we can show by contradiction that $\Delta_{cr}$ is always greater than $\Delta_{cr}^*$. Suppose that for some parameters, we have $\Delta_{cr}^* > \frac{\pi}{2\omega_{cr}}$. From the equation \eqref{omega cr}, this implies that 
\begin{eqnarray}
\frac{\pi}{2\omega_{cr}}<\Delta_{cr}^* = \frac{\delta}{1 + \delta^2\omega_{cr}^2} - \frac{\mu}{\mu^2 + \omega_{cr}^2}<\frac{\delta}{1 + \delta^2\omega_{cr}^2}\\
\frac{\pi}{2}(1 + \delta^2\omega_{cr}^2) <\delta \omega_{cr}\\
\frac{2\pi(1 - \delta^2 \mu^2)}{4 - \delta^2\lambda^2\theta^2} <  \delta \sqrt{\frac{\lambda^2\theta^2 - 4\mu^2}{4 - \delta^2\lambda^2\theta^2}} \\
4\pi^2 \cdot \frac{(1  - \delta^2\mu^2)^2}{(4 - \delta^2\lambda^2\theta^2)^2} <  \delta^2 \cdot \frac{\lambda^2\theta^2 - 4\mu^2}{4 - \delta^2\lambda^2\theta^2}\\
4\pi^2 \big(1  - \delta^2\mu^2\big)^2 <  \delta^2 (\lambda^2\theta^2 - 4\mu^2)(4 - \delta^2\lambda^2\theta^2)
\end{eqnarray}
Set $\bar{\delta} = \delta^2$. The inequality can be written as
\begin{eqnarray}
f(\bar{\delta}) = \Big(\lambda^4\theta^4 - 4\lambda^2\theta^2\mu^2 + 4\pi^2\mu^4 \Big)\bar{\delta}^2 - 4\Big(\lambda^2\theta^2 + 2\pi^2\mu^2 - 4\mu^2\Big)\bar{\delta} + 4\pi^2 <0. \label{help1}
\end{eqnarray}
Notice that the coefficient of $\bar{\delta}^2$ is always positive. It can be shown by finding $\mu^2$ that minimizes the coefficient, $\mu^2 = \frac{\lambda^2\theta^2}{2\pi^2}$, and then finding the minimum value of that coefficient, which is $\lambda^4\theta^4\Big(1 - \frac{1}{\pi^2}\Big)$ so it's clearly positive. This means that $ f(\bar{\delta})$ is a convex function, with a minimum at $\bar{\delta}^*$
\begin{eqnarray}
\bar{\delta}^* = \frac{2(\lambda^2\theta^2 + 2\pi^2\mu^2 - 4\mu^2)}{\lambda^4\theta^4 - 4\lambda^2\theta^2\mu^2 + 4\pi^2\mu^4} \\
f(\bar{\delta}) \geq f(\bar{\delta}^* ) = \frac{4(\pi^2 - 1)(\lambda^2\theta^2 - 4\mu^2)^2}{\lambda^4\theta^4 - 4\lambda^2\theta^2\mu^2 + 4\pi^2\mu^4}  >0,\label{help2}
\end{eqnarray} 
where the denominator is the same as the coefficient of $\bar{\delta}^2$ from Equation \eqref{help1}, so it must be positive. The inequalities \eqref{help1} and \eqref{help2} contradict each other, and so $\Delta_{cr}^* \leq \frac{\pi}{2\omega_{cr}}$ for all parameters. Hence by Equation \eqref{delta_cr pi/2 bound}, $\Delta_{cr}>\Delta_{cr}^*$, which implies that $c_1$ must be positive. 

Since $c_1$ is positive, the only way for $R$ from \eqref{dR deta} to have a nonzero equilibrium point is for $c_2$ to be also positive. This produces the conditions on the direction of the Hopf 
\begin{eqnarray}
\delta \mu<1 &\implies& \alpha >0 \\
\delta \mu>1 &\implies& \alpha <0 .
\end{eqnarray}
Recall that $\alpha$ represents the perturbation from $\Delta_{cr}$. So when $\delta \mu<1$, the limit cycle is born when $\Delta$ exceeds $\Delta_{cr}$. If $\delta \mu>1$, then the limit cycle is born when $\Delta$ becomes less than $\Delta_{cr}$. In either case, the equilibrium points of $R(\eta)$ are given by 
\begin{eqnarray}
R_0 = 0, \quad R_1 = \sqrt{\frac{c_2}{c_1}}>0.
\end{eqnarray}
Since $c_1,c_2,c_3>0$, $R_0$ is unstable and $R_1$ is stable. In its explicit form,
\begin{eqnarray} \label{R1}
R_1 = \sqrt{ \frac{4\alpha (\lambda^2\theta^2 - 4 \mu^2)(4 - \delta^2 \lambda^2\theta^2)^2}{\theta^2(1 - \delta^2\mu^2)(16\mu + \lambda^2\theta^2(4\Delta_{cr} -4\delta + \delta^3\lambda^2\theta^2 - 4\delta^2\mu -4\delta^2 \Delta_{cr}\mu^2))}}
\end{eqnarray}
and it represents the amplitude of the limit cycle near the Hopf. Since $R_1$ is stable, then the Hopf bifurcation is supercritical.
\end{proof}
Theorem \eqref{Theorem: Hopf stability} establishes that as $\Delta$ increases, the equilibrium becomes unstable and a stable limit cycle is born. 

\subsection{First-Order Approximation of Amplitude}
We would like to choose the weight coefficient $\delta$ in a way that minimizes the amplitude of the oscillation in queues. To do this, we first need to know what the amplitude of the oscillations is as a function of the system parameters. In the following result, we use a perturbation method to approximate the amplitude of oscillations around the bifurcation point.
\begin{proposition} \label{proposition: queue amplitude}
The amplitude of the oscillations of the queues near the first Hopf can be approximated by $\frac{R_1}{2}$, where $R_1$ is given by Equation \eqref{R1}.
\end{proposition}
\begin{proof}
The radius of the limit cycle from \eqref{R1} approximates the amplitude of the oscillations of $w_2(t)$ from \eqref{w2}. By the change of variables given in Equation \eqref{w1 w2}, as $t \to \infty$, the behavior of the queues up to a phase shift is 
\begin{eqnarray}
q_1 = \frac{1}{2}(w_1 + w_2) \to \frac{1}{2}R_1\sin(\omega \Delta t) \\
q_2 = \frac{1}{2}(w_1 - w_2) \to -\frac{1}{2}R_1\sin(\omega \Delta t).
\end{eqnarray}
Thus, the amplitude of oscillations of queues is $\frac{R_1}{2}$.
\end{proof} 

Therefore, when $\Delta$ exceeds $\Delta_{cr}$, the amplitude of oscillations can be approximated to first order by 
\begin{eqnarray}
 \label{amplitude}
\text{amplitude} \approx \sqrt{ \frac{(\Delta - \Delta_{cr}) (\lambda^2\theta^2 - 4 \mu^2)(4 - \delta^2 \lambda^2\theta^2)^2}{\theta^2(1 - \delta^2\mu^2)(16\mu + \lambda^2\theta^2(4\Delta_{cr} -4\delta + \delta^3\lambda^2\theta^2 - 4\delta^2\mu -4\delta^2 \Delta_{cr}\mu^2))}} \label{amplitude}
\end{eqnarray}
when $\Delta - \Delta_{cr}$ is small. 

The approximation is accurate when $\delta$ is substantially smaller than the ratio $\frac{2}{\lambda \theta}$. For example, in Figure \eqref{Fig: queues amp 3} the queues oscillate throughout time, and the two horizontal lines provide a good approximation of the amplitude of oscillations based on Equation \eqref{amplitude}. However, the approximation becomes inaccurate when $\delta$ approaches $\frac{2}{\lambda \theta}$. As demonstrated in Figure \eqref{Fig: queues amp 1}, when $\delta = 0.195$ and $\frac{2}{\lambda \theta} = 0.2$, the approximated amplitude is only about a half of what the actual amplitude is. The discrepancy is observed in Figures \eqref{Fig: amplitude surf} - \eqref{Fig: amplitude error d neq 0} as well. The surface plot in Figure \eqref{Fig: amplitude surf} shows the true amplitude based on numerical integration as a function of the delay $\Delta$ and the coefficient $\delta$, while the surface plot in Figure \eqref{Fig: amplitude approx 1st ord} shows the amplitude's first-order approximation. Furthermore, the surface plot in Figure \eqref{Fig: amplitude error 1st ord} shows the error of first-order approximation, where the error increases with $\delta$. Finally, Figure \eqref{Fig: amplitude error d neq 0} provides intuition for why the approximation fails as $\delta$ approaches $\frac{2}{\lambda \theta}$. Figure \eqref{Fig: amplitude error d neq 0} presents a plot comparing the amplitude and its approximation as functions of delay while $\delta = 0.19$ is close to the threshold  $\frac{2}{\lambda \theta} = 0.2$. The approximation is proportional to $ \sqrt{\Delta - \Delta_{cr}}$, while the true amplitude appears to be a linear function of $(\Delta - \Delta_{cr})$ (even though it is not exactly linear). 

\begin{figure}[H]
  \centering
  \begin{minipage}[b]{0.49\textwidth}
    \includegraphics[width=\textwidth]{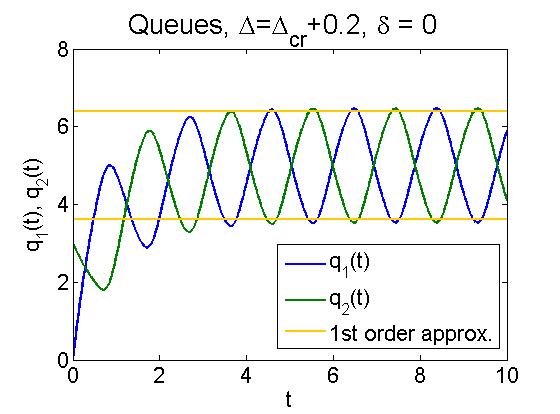}
    \caption{Amplitude approximation, \\
    $\frac{N}{\lambda \theta} = 0.2$, $\Delta = \Delta_{cr} + 0.2$, $\delta = 0$.}
    \label{Fig: queues amp 3}
  \end{minipage}
  \hfill
  \begin{minipage}[b]{0.49\textwidth}
    \includegraphics[width=\textwidth]{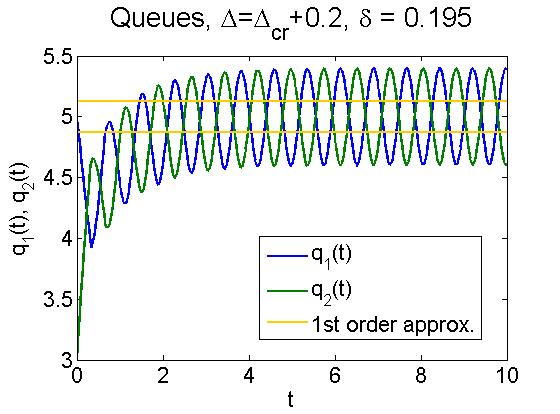}
    \caption{Amplitude approximation, \\
    $\frac{N}{\lambda \theta} = 0.2$, $\Delta = \Delta_{cr} + 0.2$, $\delta = 0.195$.}
    \label{Fig: queues amp 1}
  \end{minipage}
\end{figure}

\begin{figure}[H]
  \centering
  \begin{minipage}[b]{0.49\textwidth}
    \includegraphics[width=\textwidth]{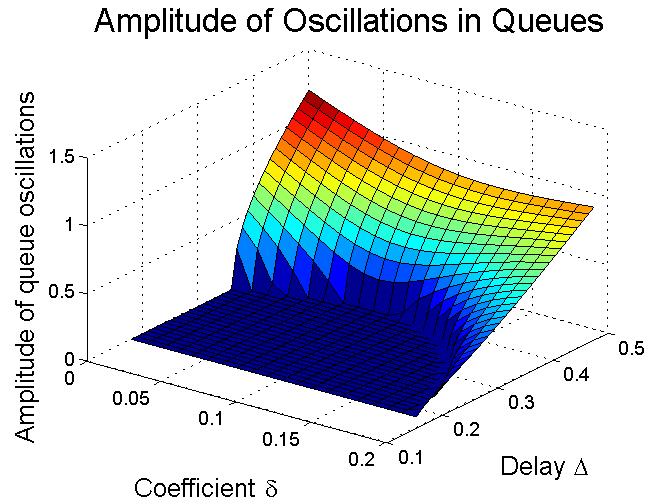}
    \caption{Amplitude of oscillations; \\
    $\theta = 1$, $\lambda = 10,$ $\mu = 1$.}
    \label{Fig: amplitude surf}
  \end{minipage}
  \hfill
  \begin{minipage}[b]{0.49\textwidth}
    \includegraphics[width=\textwidth]{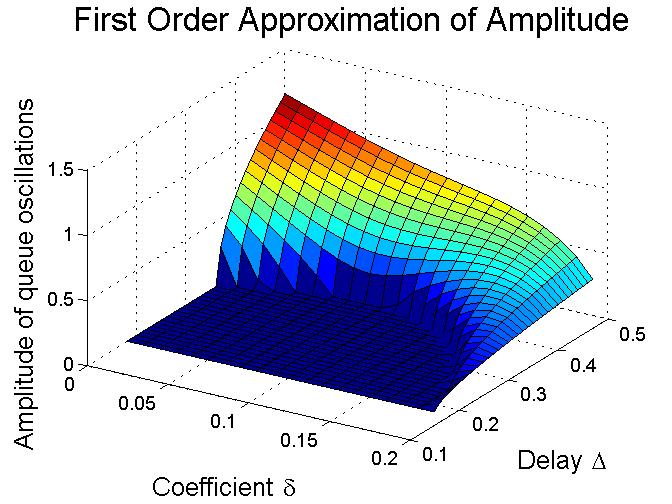}
    \caption{First-order approximation; \\
    $\theta = 1$, $\lambda = 10,$ $\mu = 1$.}
    \label{Fig: amplitude approx 1st ord}
  \end{minipage}
\end{figure}

\begin{figure}[H]
  \centering
  \begin{minipage}[b]{0.49\textwidth}
    \includegraphics[width=\textwidth]{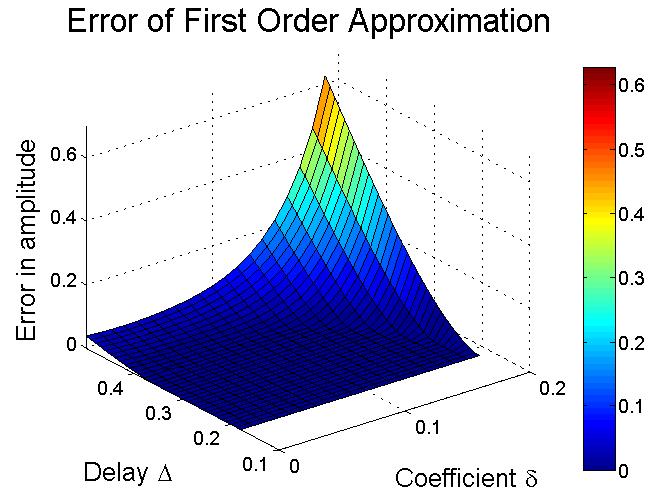}
    \caption{Error of approximation; \\
    $\theta = 1$, $\lambda = 10,$ $\mu = 1$.}
    \label{Fig: amplitude error 1st ord}
  \end{minipage}    
  \hfill
  \begin{minipage}[b]{0.49\textwidth}
    \includegraphics[width=\textwidth]{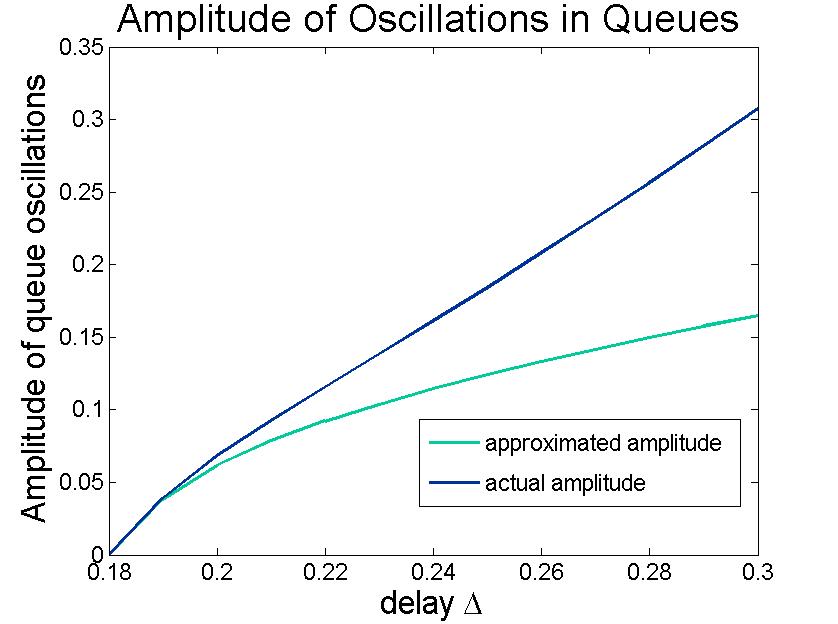}
    \caption{First-order approximation \\when $\delta = 0.19$, $\theta = 1$, $\lambda = 10,$ $\mu = 1$.}
    \label{Fig: amplitude error d neq 0}
  \end{minipage}
\end{figure}

Since we are interested in using the analytical expression of the amplitude approximation to determine the coefficient $\delta$ that minimizes the amplitude for a given delay, it is important for the approximation to be accurate. As seen from Figure \eqref{Fig: amplitude approx 1st ord}, for a fixed delay, say $\Delta = 0.5$, the point of the approximated minimum amplitude (at $\delta \approx 0.2$) does not agree with the true minimum amplitude (at $\delta \approx 0.11$). Hence, the first-order approximation of amplitude is insufficient for our purposes, and we must derive the second order of the approximation.

\subsection{Second-Order Approximation of Amplitude} \label{Sec: second order}
The first-order approximation, as seen from Equation \eqref{amplitude}, is of the form
\begin{eqnarray}
\text{Amplitude} \approx c_0 (\Delta - \Delta_{cr})^{0.5}, 
\end{eqnarray}
where $c_0$ is a factor determined by the system parameters and is independent of delay. The second-order approximation takes the form 
\begin{eqnarray}
\text{Amplitude} \approx c_0 (\Delta - \Delta_{cr})^{0.5}  + c_1(\Delta - \Delta_{cr})^{1.5},
\end{eqnarray}
where $c_1$ is also independent of the delay. The full expression for $c_1$ is long and messy, so we omit it from this Section. However, the reader can refer to the Appendix \eqref{Sec: Approximating Amplitude} for the expressions as well as a discussion on how the second-order approximation is obtained. As shown in Figures \eqref{Fig: queues amp 4} and \eqref{Fig: queues amp 2}, the second order approximation performs just as well as the first order approximation when $\delta$ is significantly smaller than $\frac{2}{\lambda \theta}$, but is much more accurate when $\delta$ approaches $\frac{2}{\lambda \theta}$. Figures \eqref{Fig: true amp} and \eqref{Fig: amplitude 2nd ord} confirm that this trend holds throughout the parameter space in $\delta$ and delay $\Delta$. Figure \eqref{Fig: amp comparison 0.1} compares the true amplitude with the two approximations when $\delta = 0.1$. The next plot in Figure \eqref{Fig: amp comparison 0.19} draws the same comparison but when $\delta = 0.19$ is closer to its upper limit $\frac{2}{\lambda \theta} = 0.2$. It is evident from the two plots that the second-order approximation is significantly more accurate than the first-order approximation, especially as $\delta \to \frac{2}{\lambda \theta}$. Figures \eqref{Fig: amplitude error 1st ord 1} and \eqref{Fig: amplitude error 2nd ord} illustrate the same point more systematically, by comparing the errors of first and second order approximations. These surface plots reveal that the higher order approximation decreases the maximum error by a factor of 10.

\begin{figure}[H]
  \centering
  \begin{minipage}[b]{0.49\textwidth}
    \includegraphics[width=\textwidth]{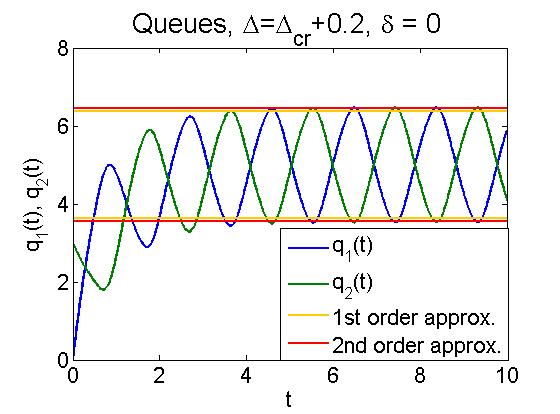}
    \caption{Amplitude approximation, \\
    $\frac{N}{\lambda \theta} = 0.2$, $\Delta = \Delta_{cr} + 0.2$, $\delta = 0$.}
    \label{Fig: queues amp 4}
  \end{minipage}
  \hfill
  \begin{minipage}[b]{0.49\textwidth}
    \includegraphics[width=\textwidth]{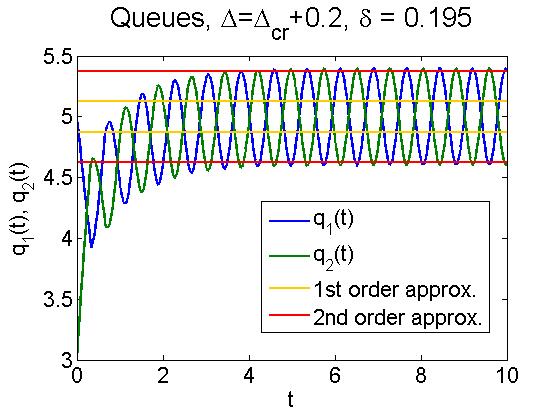}
    \caption{Amplitude approximation, \\
    $\frac{N}{\lambda \theta} = 0.2$, $\Delta = \Delta_{cr} + 0.2$, $\delta = 0.195$.}
    \label{Fig: queues amp 2}
  \end{minipage}
\end{figure}

\begin{figure}[H]
  \centering
  \begin{minipage}[b]{0.49\textwidth}
    \includegraphics[width=\textwidth]{amp_true.jpg}
    \caption{Amplitude of oscillations; \\
    $\theta = 1$, $\lambda = 10,$ $\mu = 1$. }
    \label{Fig: true amp}
  \end{minipage}    
  \hfill
  \begin{minipage}[b]{0.49\textwidth}
    \includegraphics[width=\textwidth]{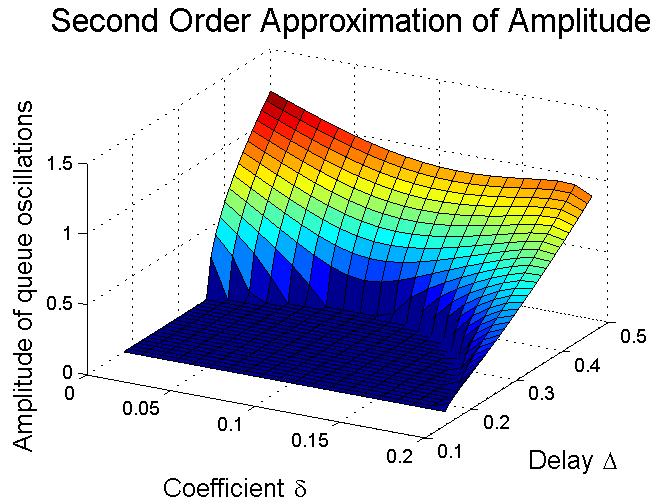}
    \caption{First-order approximation; \\
    $\theta = 1$, $\lambda = 10,$ $\mu = 1$. }
    \label{Fig: amplitude 2nd ord}
  \end{minipage}
\end{figure}

\begin{figure}[H]
  \centering
  \begin{minipage}[b]{0.49\textwidth}
    \includegraphics[width=\textwidth]{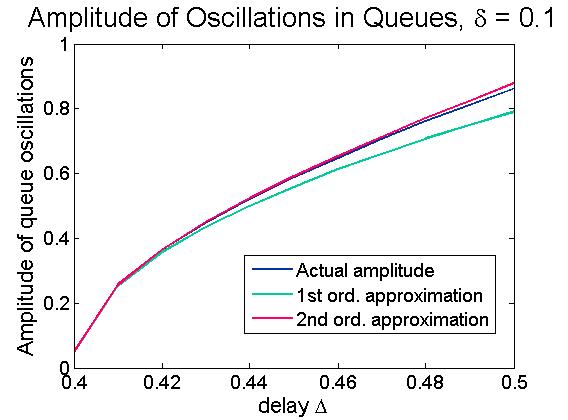}
    \caption{Comparison when $\delta = 0.10$; \\
    $\theta = 1$, $\lambda = 10,$ $\mu = 1$. }    \label{Fig: amp comparison 0.1}
  \end{minipage}    
  \hfill
  \begin{minipage}[b]{0.49\textwidth}
    \includegraphics[width=\textwidth]{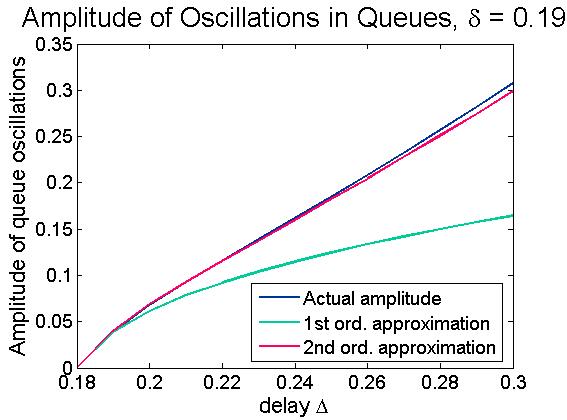}
    \caption{Comparison when $\delta = 0.19$; \\
    $\theta = 1$, $\lambda = 10,$ $\mu = 1$. }    \label{Fig: amp comparison 0.19}
  \end{minipage}
\end{figure}

\begin{figure}[H]
  \centering
  \begin{minipage}[b]{0.49\textwidth}
    \includegraphics[width=\textwidth]{amp_err_1st_ord.jpg}
        \caption{First-order error; \\
    $\theta = 1$, $\lambda = 10,$ $\mu = 1$.}
    \label{Fig: amplitude error 1st ord 1}
  \end{minipage}    
  \hfill
  \begin{minipage}[b]{0.49\textwidth}
    \includegraphics[width=\textwidth]{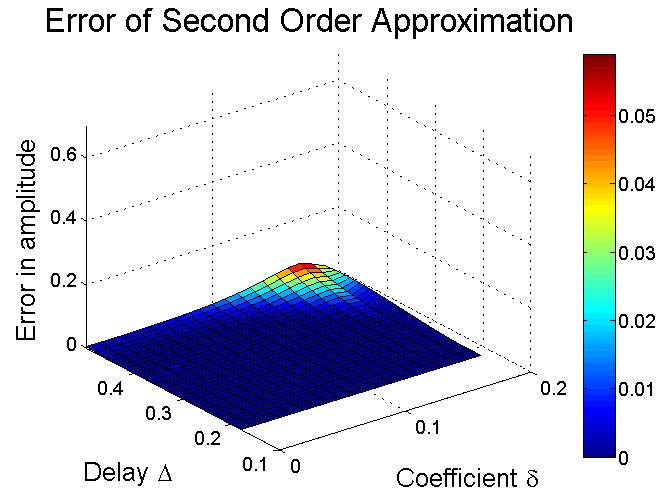}
    \caption{Second-order error; \\
    $\theta = 1$, $\lambda = 10,$ $\mu = 1$.}
    \label{Fig: amplitude error 2nd ord}
  \end{minipage}
\end{figure}
\subsection{Minimizing the Amplitude of Oscillations} \label{Sec: minimize amplitude}
Since the second-order approximation is sufficiently accurate, we proceed by using the analytical formula of the second-order approximation to determine the coefficient $\delta$ that minimizes the amplitude of oscillations. Figure \eqref{Fig: amplitude minimized} shows the numerically computed amplitude, together with its minimum for each delay according to the second-order approximation. The minimum of the amplitude as a function of $\delta$ is found numerically in MATLAB. It is evident that the approximated minimum closely corresponds to where the true minimum is. 

Figure \eqref{Fig: amplitude minimized} shows that the velocity information indeed affects the amplitude of oscillations, and the amplitude can be reduced with a proper choice of the coefficient $\delta$. Figure \eqref{Fig: amplitude minimized} also reveals an important finding. {\bf The value $\delta_{max}$ for the coefficient $\delta$ that maximizes $\Delta_{cr}$ is not the same as $\delta_{amp}$ that minimizes the amplitude of oscillations}. Specifically, $\delta_{max}$ is independent of the delay $\Delta$, while $\delta_{amp}$ is a function of the delay. The one point where the two values are guaranteed to be equal each other, $\delta_{max} = \delta_{amp}$, is when $\delta_{amp}$ is computed for the delay eqal to the maximum possible $\Delta_{cr}$, i.e. $\Delta = \Delta_{cr}(\delta_{max})$. Therefore, one should use $\delta_{max}$ as the weight coefficient as long as the delay is less than the bifurcation threshold $\Delta_{cr}$ evaluated at $\delta_{max}$, but when the delay exceeds $\Delta_{cr}$ one should use $\delta_{amp}$ for the weight coefficient instead.

\begin{figure}[H]
  \centering
    \includegraphics[width=0.75\textwidth]{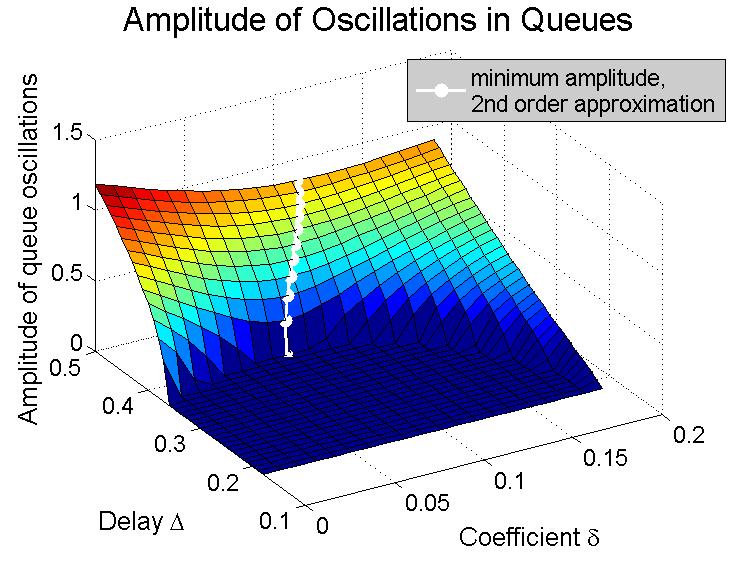}
        \caption{For any delay, the amplitude can be minimized as a function of $\delta$.}
    \label{Fig: amplitude minimized}
\end{figure}


\section{Conclusion}
This paper answers important questions with regards to businesses incorporating the queue length velocity into the information that is provided to customers via delay announcements. We consider the information passed to the customers about each of $N$ queues to be a linear combination of the current queue length and the rate at which that queue is moving, or the queue velocity,
\begin{eqnarray}
\text{delay announcement about }i^{th}\text{ queue} = q_i(t- \Delta) + \delta \shortdot{q}_i(t - \Delta) ,
\end{eqnarray}
with the delay $\Delta$ being the time of customers travelling to the selected queue. 

The most evident finding is that the coefficient $\delta$ that weighs the queue velocity information should always be less than the ratio $\frac{N}{\lambda \theta}$. Maintaining this limit guarantees that, at best, the queues will be locally stable for any delay in information. At worst, the queues will be stable when the delay $\Delta$ is sufficiently small, eventually undergoing a Hopf bifurcation at $\Delta = \Delta_{cr}$ and becoming unstable. Alternatively, if $\delta >\frac{N}{\lambda \theta}$, then the queues will never be stable even when the delay in information is infinitesimally small. The reader can refer to Figure \eqref{Fig: queue stability cases 0} for more details. 

Even when the condition $\delta < \frac{N}{\lambda \theta}$ is met, significant improvements can still be made by choosing $\delta$ optimally. In the case when queues become unstable as the delay in information increases (so when $\lambda \theta > N\mu$), the weight $\delta$ can shift the delay threshold $\Delta_{cr}$ at which the queues become unstable. In fact, there is exists a "cap" on the weight $\delta_{cap}$, such that it is safe and beneficial to include the queue velocity whenever $\delta \leq \delta_{cap}$, meaning that the queues will remain stable under greater delay than if the velocity information was omitted. Further, if the threshold $\delta_{cap}$ is exceeded, then queue velocity information will be harmful to the system. In this case, the queues will lose their stability for a smaller delay $\Delta$ than if the queue velocity was omitted altogether. An edge case that exemplifies the usefulness of this discovery is as follows. If we take $\delta \to \frac{N}{\lambda \theta}$, at which point is is clear that $\delta > \delta_{cap}$, then the queues bifurcate almost immediately because $\Delta_{cr} \to 0$ even though the same queues would have remained stable under a much larger delay if $\delta$ was set to 0. Hence, it is important to keep $\delta$ smaller than $\delta_{cap}$.

We also showed that there exists an optimal value for $\delta$ called $\delta_{max}$ that gives the most stability to the queues. For $\delta = \delta_{max}$, queues will be stable for greater delay than is possible given any other choice of $\delta$. We provide an equation from which $\delta_{max}$ can be found numerically, as well as closed form expressions for upper and lower bounds on $\delta_{max}$. Choosing $\delta$ within those bounds is a safe choice for the service managers.

This leads to a natural assessment of the limitations of providing the queue velocity information. The threshold $\Delta_{cr}$ where the queues lose stability can be arbitrarily close to 0 when $\delta$ is chosen poorly, but even the best choice of $\delta$ can only help so much. We provide a formula from which the maximum attainable $\Delta_{cr}$ can be computed. Further, we give expressions on the bounds for that optimal $\Delta_{cr}$ because they don't rely on $\delta_{max}$ and hence they may be easier to evaluate. This means that while including $\delta$ can always improve the queue dynamics to some degree, there is a limit on how much impact $\delta$ may have. 

The presence of the queue velocity information can also affect the amplitude with which the queues oscillate after losing stability. From numerical integration of the queues such as in Figure \eqref{Fig: amplitude surf}, it is clear that incorporating the queue velocity information can decrease the amplitude of the oscillations, which is beneficial from the managerial perspective. Using a perturbation technique, we derive an analytic expression that approximates the amplitude of oscillations very accurately. Based on the analytic expression, for any delay we can determine the coefficient $\delta$ that will minimize the amplitude of the oscillations. We note that this coefficient as a function of delay, and is not necessarily equal to the coefficient $\delta_{max}$ that maximizes the delay threshold.

In the future, it would be interesting to extend this model to include terms with higher order derivatives. Under the assumption that service managers can measure the information about the queues  ($q_i^{(2)},q_i^{(3)},\dots$), it would be natural to incorporate this data into the information that is provided to the customers:
\begin{eqnarray}
\text{delay announcement about }i^{th}\text{ queue} = q_i(t- \Delta) + \sum_{n = 1}^{K} \delta_n q_i^{(n)}(t - \Delta), \quad K \in \mathbb{N}.
\end{eqnarray}
The equations describing such a queueing system will no longer be neutral, and may be more complicated. However, such queueing system may answer new questions. One question of significance is to determine the minimum sufficient number of higher-order derivatives ($K$) that should be included in order to guarantee that the queues will be stable for a given delay.


\bibliographystyle{ieeetr}
\bibliography{main}


\section{Appendix}
\subsection{Uniqueness and existence of the equilibrium}\label{Subsec: equilibrium existence}
\begin{proof}[Proof of Theorem \eqref{Thm: equilibrium existence}] 
To check that $q_i(t)= \frac{\lambda}{N\mu}$ is an equilibrium, plug into Equation \eqref{q_i dot} to get
\begin{equation}
\longdot{q}_i(t) = \lambda \cdot \frac{\exp(-\frac{\lambda\theta}{N\mu} - 0)}{\sum^{N}_{j=1} \exp(-\frac{\lambda\theta}{N\mu} - 0) } - \mu \cdot \frac{\lambda}{N\mu} = \frac{\lambda}{N} - \frac{\lambda}{N} = 0.
\end{equation}

To show uniqueness, we will argue by contradiction. Suppose there is another equilibrium given by $\bar{q}_i$, $1\leq i \leq N$, and for some $i$ we have $q^*_i \neq \bar{q}_i$. The following condition must hold
\begin{eqnarray}
0 = \sum_{i = 1}^N\longdot{q}_i(t) = \lambda \cdot \frac{\sum_{i = 1}^N \exp\big(-\theta\bar{q}_i(t - \Delta)\big)}{\sum_{j = 1}^N \exp(-\theta\bar{q}_j(t - \Delta))} - \mu \sum_{i=1}^N \bar{q}_i(t), \quad \sum_{i=1}^N \bar{q}_i(t) = \frac{\lambda}{\mu} .\label{sum of queues}
\end{eqnarray}
Hence, the mean of $\bar{q}_i$ is $\frac{\lambda}{N\mu} $ and since $\bar{q}_i$ cannot all be equal to each other, there must exist some $\bar{q}_s$ that is smaller than the mean, and some $\bar{q}_g$ that is greater than the mean
\begin{equation}
\label{q bar small}
\bar{q}_s = \frac{\lambda}{\mu N} - \gamma, \qquad \bar{q}_b = \frac{\lambda}{N\mu} + \epsilon, \qquad \gamma, \epsilon > 0.
\end{equation}
This leads to a contradiction:
\begin{equation}
\longdot{q}_s(t) =  \lambda \frac{\exp\big(-\theta\bar{q}_s\big)}{\sum_{i = 1}^N\exp\big(-\theta\bar{q}_i \big)} - \mu \bar{q}_s  =0 
\end{equation}
\begin{equation} 
\implies \sum_{i = 1}^N\exp\big(-\theta\bar{q}_i \big) = \frac{\lambda}{\mu}\cdot\frac{\exp\big(-\frac{\theta\lambda}{N\mu} + \theta\gamma \big)}{(\frac{\lambda}{N\mu} - \gamma)},\label{sum of exp} 
\end{equation}
\begin{eqnarray}
\longdot{q}_g(t) = \lambda \frac{\exp\big(-\theta \bar{q}_g\big)}{\sum_{i = 1}^N\exp\big(-\theta \bar{q}_i \big)} - \mu \bar{q}_g (t) \\
= \lambda \frac{\exp\big(-\frac{\theta\lambda}{N\mu} - \theta\epsilon\big)}{\frac{\lambda}{\mu}\cdot\frac{\exp(-\frac{\theta\lambda}{N\mu} + \theta\gamma)}{(\frac{\lambda}{N\mu} - \gamma)}} - \mu \Big( \frac{\lambda}{N\mu} + \epsilon \Big) \\
= -\frac{\lambda}{N}\big(1 - e^{-\theta(\epsilon + \gamma)}\big) - \mu \big(\epsilon + \gamma e^{-\theta(\epsilon + \gamma)}\big) <0.
\end{eqnarray}
Since $\longdot{q}_g(t) \neq 0$, then $\bar{q}_i(t)$ is not an equilibrium, and the equilibrium \eqref{equilibrium} is unique.  
\end{proof}

\subsection{Approximation to Amplitude of Oscillations in Queues} \label{Sec: Approximating Amplitude} 
To see how the velocity information affects the behavior of the queues after a Hopf bifurcation occurs, we need to develop approximations for the amplitude of oscillations. In Section \eqref{Sec: N=2}, we find a first-order approximation to the amplitude but observe that it is not sufficiently accurate. Hence, we require a second-order approximation. The steps to determing the second-order approximation are outlined below. 

This process is very closely related to the staps taken in Theorem \eqref{Theorem: Hopf stability}. We begin with Equation \eqref{x}, and expand the time $\tau = \omega t$. Then expand our functions of interest in $\epsilon$ to the second order:
\begin{eqnarray}
x(\tau) = x_0(\tau) + \epsilon x_1 (\tau) + + \epsilon^2 x_2 (\tau),\quad  \Delta = \Delta_0 + \epsilon \Delta_1 + \epsilon^2 \Delta_2, \quad  \omega = \omega_0 + \epsilon \omega_1 + \epsilon^2 \omega_2, \nonumber
\end{eqnarray}
where $\Delta_0$ and $\omega_0$ are the delay and frequency at bifurcation, so $\Delta_0 = \Delta_{cr}$ and $\omega_{cr}$. By collecting all the terms with the like powers of $\epsilon$ into separate equations, we get equations from which we can solve for $x_0$ and $x_1$. From the equation for $\epsilon^0$ we find that $x_0(\tau) = A \cos(\tau)$ is a solution. Next, we use the equation for $\epsilon^1$ terms to solve for $A$, which has the expression given by Equation \eqref{R1}. We can now find $x_1$ that has a solution of the form $x_1(\tau) = a_1 \sin(\tau) + a_2 \cos(\tau) + a_3 \sin(3 \tau) + a_4 \cos(3 \tau)$. The coefficients $a_3$ and $a_4$ are determined from equation for $\epsilon^1$ terms. We impose the initial condition $x'(0) = 0$ to ensure that the maximum amplitude is at $0$, which implies $a_1 = - 3 a_3$. Lastly, we determine $a_2$ by eliminating the secular terms from the equation for $\epsilon^2$ terms. Therefore, the second-order approximation of the amplitude of oscillations can be deduced from 
\begin{eqnarray}
x(\tau) \approx x_0(\tau) + \epsilon x_1(\tau) \\
= A \cos(\tau) + \epsilon\big( a_1 \sin(\tau) + a_2 \cos(\tau) + a_3 \sin(3 \tau) + a_4 \cos(3 \tau) \big),
\end{eqnarray}
where the coefficients given below:
\begin{eqnarray} 
A = \sqrt{ \frac{4 \Delta_1 (\lambda^2\theta^2 - 4 \mu^2)(4 - \delta^2 \lambda^2\theta^2)^2}{\theta^2(1 - \delta^2\mu^2)(16\mu + \lambda^2\theta^2(4\Delta_0 -4\delta + \delta^3\lambda^2\theta^2 - 4\delta^2\mu -4\delta^2 \Delta_0\mu^2))}} \nonumber
\end{eqnarray}
\begin{eqnarray}
\omega_1 = \frac{4 \Delta_1 \theta ^2 \lambda ^2 \left(\delta ^2 \mu ^2-1\right) \sqrt{\theta ^2 \lambda ^2-4 \mu ^2}}{\sqrt{4-\delta ^2 \theta ^2 \lambda ^2} \left(\theta ^2 \lambda ^2 \left(\delta  \left(\delta ^2 \theta ^2 \lambda ^2-4 \delta  \mu  (\text{$\Delta $0} \mu +1)-4\right)+4 \text{$\Delta $0}\right)+16 \mu \right)} \nonumber
\end{eqnarray}
\begin{eqnarray}
a_1 = - 3 a_3 = -\Big(2 A^3 \theta ^2 \omega_0^3 \left(\theta ^2 \lambda ^2 \mu  \left(\delta ^2 \omega_0^2+1\right)^3-4 \delta ^3 \left(\mu ^2+\omega_0^2\right)^3\right)\Big)\nonumber  \\
/\Big(\theta ^4 \lambda ^4 \left(\delta ^2 \omega_0^2+1\right)^3 \left(\mu ^2+9 \omega_0^2\right)+16 \left(9 \delta ^2 \omega_0^2+1\right) \left(\mu ^2+\omega_0^2\right)^3 \nonumber \\
+8 \theta ^2 \lambda ^2 \big(-9 \delta ^4 \omega_0^8-6 \mu ^2 \omega_0^2 \left(\delta ^2 \mu ^2+1\right)+2 \delta ^2 \omega_0^6 (\delta  \mu  (9 \delta  \mu -32)+9) \nonumber \\
+3 \omega_0^4 \left(\delta ^4 \mu ^4-12 \delta ^2 \mu ^2+1\right)-\mu ^4\big)\Big)\nonumber 
\end{eqnarray}
\begin{eqnarray}
a_4 = -\frac{1}{12}\Big(A^3 \theta ^2 \left(\theta ^2 \lambda ^2 \left(\delta ^2 \omega_0^2+1\right)^3 \left(\mu ^4+6 \mu ^2 \omega_0^2-3 \omega_0^4\right)+4 \left(3 \delta ^4 \omega_0^4-6 \delta ^2 \omega_0^2-1\right) \left(\mu ^2+\omega_0^2\right)^3\right)\Big) \nonumber \\
/\Big(\theta ^4 \lambda ^4 \left(\delta ^2 \omega_0^2+1\right)^3 \left(\mu ^2+9 \omega_0^2\right)+16 \left(9 \delta ^2 \omega_0^2+1\right) \left(\mu ^2+\omega_0^2\right)^3 \nonumber \\
+8 \theta ^2 \lambda ^2 \big(-9 \delta ^4 \omega_0^8-6 \mu ^2 \omega_0^2 \left(\delta ^2 \mu ^2+1\right)+2 \delta ^2 \omega_0^6 (\delta  \mu  (9 \delta  \mu -32)+9) \nonumber \\
+3 \omega_0^4 \left(\delta ^4 \mu ^4-12 \delta ^2 \mu ^2+1\right)-\mu ^4\big)\Big)\nonumber 
\end{eqnarray}
\begin{eqnarray}
a_2 = \frac{1}{12}\Big(A^5 \theta ^4 \left(\delta ^2 \omega_0^2+1\right)^2 \left(\delta ^2 \mu  \omega_0^2+\mu ^2 \left(\delta  \left(\delta  \Delta_0 \omega_0^2-1\right)+\Delta_0\right)+\omega_0^2 \left(\delta  \left(\delta  \Delta_0 \omega_0^2-1\right)+\Delta_0\right)+\mu \right) \nonumber \\%
-12 A^3 \theta ^2 \omega_0 \big(\omega_1 \left(\delta  \left(3 \delta ^3 \Delta_0 \omega_0^4+\delta  \omega_0^2 (\delta  (\delta  \mu  (2 \Delta_0 \mu +3)-3)+4 \Delta_0)+\delta  \mu  (-2 \delta  \mu +2 \Delta_0 \mu +3)-1\right)+\Delta_0\right) \nonumber \\%
-\Delta_1 \omega_0 \left(\delta ^2 \mu ^2-1\right) \left(\delta ^2 \omega_0^2+1\right)\big) +12 A^2 \theta ^2 \big(a_1 \omega_0 \left(\delta ^2 \mu ^2-1\right) \left(\delta ^2 \omega_0^2+1\right) \nonumber \\
+a_3 \omega_0 \left(\delta ^2 \left(\omega_0^2 (\delta  \mu  (-3 \delta  \mu +8 \Delta_0 \mu +8)-5)+8 \delta  \Delta_0 \omega_0^4+\mu ^2\right)-1\right) \nonumber \\
+a_4 (3 \delta ^4 \Delta_0 \omega_0^6+3 \delta ^2 \omega_0^4 (\delta  (\delta  \mu  (\Delta_0 \mu +1)-1)-2 \Delta_0)+\omega_0^2 (\delta  (\delta  \mu  (5 \delta  \mu -6 \Delta_0 \mu -6)+1)-\Delta_0) \nonumber \\
+\mu  (\delta  \mu -\Delta_0 \mu -1))\big) -96 A \big(2 \Delta_1 \omega_0 \omega_1 (\mu  \left(\mu  \left(2 \delta ^2-2 \delta  \Delta_0+\Delta_0^2\right)-\delta +\Delta_0\right)+\delta ^2 \Delta_0^2 \omega_0^4 \nonumber \\
+\Delta_0 \omega_0^2 (\delta  (\delta  \mu  (\Delta_0 \mu +1)-2)+\Delta_0)-1) +\Delta_0 \omega_1^2 (\delta ^2 \Delta_0^2 \omega_0^4+\Delta_0 \omega_0^2 (\delta  (\delta  \mu  (\Delta_0 \mu +1)-3)+\Delta_0) \nonumber \\
+\mu  (2 \delta -\Delta_0) (\delta  \mu -\Delta_0 \mu -1))+\Delta_1^2 \omega_0^2 (\delta ^2 \Delta_0 \omega_0^4+\omega_0^2 (\delta  (\delta  \mu  (\Delta_0 \mu +1)-1)+\Delta_0) \nonumber \\
+\mu  (-\delta  \mu +\Delta_0 \mu +1))\big)-192 a_1 \big(\omega_1 (\delta ^2 \Delta_0^2 \omega_0^4 
+\Delta_0 \omega_0^2 (\delta  (\delta  \mu  (\Delta_0 \mu +2)-2)+\Delta_0)\nonumber \\
+(-\delta  \mu +\Delta_0 \mu +1)^2)+\Delta_1 \omega_0 \left(\delta ^2 \Delta_0 \omega_0^4+\omega_0^2 (\delta  (\delta  \mu  (\Delta_0 \mu +1)-1)+\Delta_0)+\mu  (-\delta  \mu +\Delta_0 \mu +1)\right)\big)\Big) \nonumber \\
/\Big(3 A^2 \theta ^2 \left(\delta ^2 \omega_0^2+1\right) \left(\delta ^2 \Delta_0 \omega_0^4+\omega_0^2 (\delta  (\delta  \mu  (\Delta_0 \mu +1)-1)+\Delta_0)+\mu  (-\delta  \mu +\Delta_0 \mu +1)\right) \nonumber \\
+16 \Delta_1 \omega_0^2 \left(\delta ^2 \mu ^2-1\right)\Big). \nonumber 
\end{eqnarray}

To reproduce our numerical results from Section \eqref{Sec: second order} - \eqref{Sec: minimize amplitude}, set $\epsilon = 1$ and $\Delta_1 = \frac{1}{\epsilon}(\Delta - \Delta_0)$, with $\Delta_0$ given by Equation \eqref{delta cr}. Note that in the equations above there is no presence of $\Delta_2$, because we have have set $\Delta_2 = 0$. There is no equation that determines $\Delta_2$ and $\Delta_1$ uniquely, and the only restriction is that $\Delta = \Delta_0 + \epsilon \Delta_1 + \epsilon^2 \Delta_2$. Prior to choosing $\Delta_2 $ to be 0, we experimented numerically with different combinations of $\Delta_1$ and $\Delta_2$, and determined that the pair $\Delta_1 = \frac{1}{\epsilon}(\Delta - \Delta_0)$ and $\Delta_2 = 0$ results in nearly the most accurate approximation.

\end{document}